\documentclass[10pt]{amsart}
    \pdfoutput=1
\usepackage[alphabetic]{amsrefs}
    \usepackage{graphicx}
    \usepackage[english]{babel}
    \usepackage{graphicx}
    \usepackage{framed}
    \usepackage[normalem]{ulem}
    \usepackage{amsmath}
    \usepackage{dynkin-diagrams}
    \usepackage{amsthm}
    \usepackage{amssymb}
    \usepackage{hyperref}
    \usepackage[capitalize]{cleveref}
    \usepackage{comment}
    \usepackage{tikz}
    \usetikzlibrary{matrix,calc}
    \usepackage{mathtools}
    \usepackage{tikz-cd}
    \usepackage{amsfonts}
    \usepackage{enumerate}
    \usepackage[utf8]{inputenc}
    \usepackage[top=1 in,bottom=1in, left=1 in, right=1 in]{geometry}
      \usepackage{stmaryrd}

    \usepackage[colorinlistoftodos]{todonotes}
    \usepackage[all]{xy}
    \usepackage{mathrsfs}

    \newcommand{\Q}{\mathbb{Q}} 
    \newcommand{\A}{\mathbb{A}} 
    \newcommand{\G}{\mathbb{G}} 
    
    \newcommand{\SL}{\operatorname{SL}}

    \newcommand{\LG}{\mathfrak{g}}
    
    \newcommand{\LT}{\mathfrak{t}}

    \newcommand{\Hom}{\operatorname{Hom}}
    
    \renewcommand{\Hom}{\operatorname{Hom}} 
    \newcommand{\uHom}{\underline{\operatorname{Hom}}} 
    \newcommand{\uEnd}{\underline{\operatorname{End}}} 
    \newcommand{\QCoh}{\operatorname{QCoh}}
    \newcommand{\IndCoh}{\operatorname{IndCoh}}
    
    \newcommand{\C}{\mathcal{C}} 
    \renewcommand{\O}{\mathcal{O}} 
    \newcommand{\F}{\mathcal{F}} 
    \newcommand{\D}{\mathcal{D}} 
    \renewcommand{\\}{\backslash}

    \theoremstyle{definition}
    \newtheorem{Theorem}{Theorem}[section]
    
    \newtheorem{Corollary}[Theorem]{Corollary}
    \newtheorem{Definition}[Theorem]{Definition}
    \newtheorem{Proposition}[Theorem]{Proposition}
    \newtheorem{Remark}[Theorem]{Remark}
    \newtheorem{Example}[Theorem]{Example}
    \newtheorem{Lemma}[Theorem]{Lemma}

    \title{The Coarse Quotient for Affine Weyl Groups and Pseudo-reflection Groups}
    \author{Tom Gannon}

    \newcommand{\AvN}{\text{Av}_*^N}
    \newcommand{\Avpsi}{\text{Av}_{!}^{\psi}}
    \newcommand{\Symt}{\text{Sym(}\mathfrak{t}\text{)}}

    \newcommand{\LTd}{\LT^{\ast}}
    \newcommand{\Symtplus}{\text{Sym}(\mathfrak{t})^+}
    \newcommand{\Symtt}{\text{Sym}(\mathfrak{t} \times \mathfrak{t})}
    \newcommand{\Wext}{\tilde{W}^{\text{aff}}}
    \newcommand{\Waff}{W^{\text{aff}}}

\begin{document}
    
\renewcommand{\G}{\mathcal{G}}
\newcommand{\AvNTw}{\text{Av}_*^{N, (T, w)}}
\newcommand{\AvGw}{\text{Av}_{\ast}^{G,w}}
\newcommand{\pifin}{\pi_{\text{fin}}}
\newcommand{\pifinL}{\pi_{\text{fin,L}}}

    \newcommand{\ELeftAdjoint}{\text{ev}_{\omega_{\LTd}}}
\newcommand{\ClGlobalDiffOp}{\text{H}^0\Gamma(\mathcal{D}_{G/N})}
\newcommand{\GlobalDiffOp}{\Gamma(\mathcal{D}_{G/N})}
\newcommand{\indsch}{\mathcal{X}}

    \newcommand{\DGCatContk}{\text{DGCat}^k_{\text{cont}}}
\newcommand{\DGCatContL}{\text{DGCat}^L_{\text{cont}}}

\newcommand{\DNTw}{\mathcal{D}(N\backslash G/N)^{T_r,w}}
\newcommand{\DNTwWhit}{\mathcal{D}(N^-_{\psi}\backslash G/N)^{T_r,w}}
\newcommand{\DNWhit}{\mathcal{D}(N^-_{\psi}\backslash G/N)}
\newcommand{\DNTwldeg}{\mathcal{D}(N \backslash G/N)^{T_r,w}_{\text{left-deg}}}
\newcommand{\DNTwnondeg}{\mathcal{D}(N \backslash G/N)^{T_r,w}_{\text{nondeg}}}
\newcommand{\DN}{\mathcal{D}(N\backslash G/N)}
\newcommand{\DNldeg}{\mathcal{D}(N \backslash G/N)_{\text{left-deg}}}
\newcommand{\DNnondeg}{\mathcal{D}(N \backslash G/N)_{\text{nondeg}}}
\newcommand{\DNlambda}{\mathcal{D}^{\lambda}(N \backslash G/B)}
\newcommand{\Dpsilambda}{\mathcal{D}^{\lambda}(N^- _{\psi}\backslash G/B)}
\newcommand{\DbiTw}{\mathcal{D}(N \backslash G/N)^{T \times T, w}}
\newcommand{\DbiTwnondeg}{\mathcal{D}(N \backslash G/N)^{T \times T, w}_{\text{nondeg}}}
\newcommand{\DbiTwnondegheart}{\mathcal{D}(N \backslash G/N)^{(T \times T, w), \heartsuit}_{\text{nondeg}}}
\newcommand{\DbiTwdeg}{\mathcal{D}(N \backslash G/N)^{T \times T, w}_{\text{deg}}}
\newcommand{\DNBlambda}{\mathcal{D}(N \backslash G/_{\lambda}B)}
\newcommand{\DNTwBlambda}{\mathcal{D}(N \backslash G/_{\lambda}B)}
\newcommand{\DWhitBlambda}{\mathcal{D}(N^-_{\psi}\backslash G/_{\lambda}B)}
\newcommand{\HN}{\D(N \backslash G/N)}
\newcommand{\HNTw}{\D(N \backslash G/N)^{T \times T, w}}
\newcommand{\HNTwabbreviated}{\mathcal{H}^{N, (T,w)}}
\renewcommand{\indsch}{\mathcal{X}}
\newcommand{\wdot}{\dot{w}}
\newcommand{\Gaminusalpha}{\mathbb{G}_a^{-\alpha}}
\newcommand{\Aone}{\mathbf{A}}
\newcommand{\Atwo}{\mathcal{L}\text{-mod}(\Aone)}
\newcommand{\algobj}{\mathcal{A}}
\newcommand{\newalgobj}{\mathcal{A}'}
\newcommand{\tilder}{\tilde{r}}
\newcommand{\Ccirc}{\mathring{\C}}
\newcommand{\rootlattice}{\mathbb{Z}\Phi}
\newcommand{\characterlatticeforT}{X^{\bullet}(T)}
\newcommand{\Spec}{\text{Spec}}
\newcommand{\FourierMukai}{\text{FMuk}}
\newcommand{\oneshiftedCartierdual}{c_1}
\newcommand{\quotientmapforcoarsequotient}{\overline{s}}

\newcommand{\tildeV}{\tilde{\mathbb{V}}}
\newcommand{\Vdual}{V^{\vee}}
\newcommand{\vectorspaceHactson}{V}
\newcommand{\generalstacktoGITquotientmap}{\phi}
\newcommand{\SpecofL}{\text{Spec}(L)}
\newcommand{\terminalmapfromC}{\alpha}
\newcommand{\terminalmapfromCmodassociatedstabilizer}{\dot{\alpha}}

\newcommand{\AvNshifted}{\AvN[\text{dim}(N)]}
\newcommand{\hyperplanefixedbys}{V^{\ast}_{s = \text{id}}}
\newcommand{\pointwisegeneralstacktoGITquotientmap}{\dot{\generalstacktoGITquotientmap}}
\newcommand{\fieldpossiblydifferentfromgroundfield}{K}
\newcommand{\Avpsishifted}{\Avpsi[-\text{dim}(N)]}

    \maketitle
\begin{abstract}
We study the coarse quotient $\LTd\sslash\Waff$ of the affine Weyl group $\Waff$ acting on a dual Cartan $\LTd$ for some semisimple Lie algebra. Specifically, we classify sheaves on this space via a \lq pointwise\rq{} criterion for descent, which says that a $\Waff$-equivariant sheaf on $\LTd$ descends to the coarse quotient if and only if the fiber at each field-valued point descends to the associated GIT quotient.

We also prove the analogous pointwise criterion for descent for an arbitrary finite group acting on a vector space.  Using this, we show that an equivariant sheaf for the action of a finite pseudo-reflection group descends to the GIT quotient if and only if it descends to the associated GIT quotient for every pseudo-reflection, generalizing a recent result of Lonergan.

\end{abstract}
\tableofcontents

 \section{Introduction}
\label{Overview of the Main Results}
The goal of this paper is to study the \textit{coarse quotient} $\LTd\sslash\Waff$ for use in \cite{GannonCategoricalRepresentationTheoryandTheCoarseQuotient}. While this quotient appears naturally in many representation theoretic contexts (see \cref{Motivation Subsection}), its global geometry makes direct analysis of this space difficult--for example, it is not a scheme or algebraic space, and its diagonal map is not quasicompact. However, in this paper we argue that many questions about $\LTd\sslash\Waff$ and its category of sheaves can be understood through the use of \textit{field-valued points}. Specifically, we show the following, which we also use to re-derive the main results of \cite{LoRemark}:

\begin{Theorem} (\cref{CoherentFullyFaithful}, \cref{Various Conditions for Wext Equivariant Sheaf to Satisfy Coxteter Descent})
The pullback map induced by $\LTd/\Waff \to \LTd\sslash \Waff$ is fully faithful. Moreover, a $\Waff$-equivariant sheaf $\F$ on $\LTd$ lies in the essential image of this pullback if and only if at every field-valued point $x$ of $\LTd$, the $W_x$-representation on the fiber of $\F$ at $x$ is trivial. 
\end{Theorem}

Field-valued points can also be used to better understand the coarse quotient $\vectorspaceHactson\sslash H := \Spec(\text{Sym}(\Vdual)^H)$ in the classical setting of a finite group $H$ acting on a finite dimensional $k$-vector space $V$. For example, using field-valued points, we generalize a result of Lonergan \cite{LoRemark} on descent to the coarse quotient for finite Coxeter groups to the setting of \textit{pseudo-reflection groups}, see \cref{Equivariant Sheaf for Pseudo-Reflection Group Descends iff It Descends for Each Reflection}. After giving the definition of the coarse quotient in \cref{The Coarse Quotient Intro Subsection}, we discuss these applications in \cref{Descent to the Coarse Quotient for Finite Pseudo-Reflection Groups Subsection}, and then give some motivation for the study of the coarse quotient in \cref{Motivation Subsection}. 
\subsection{The Coarse Quotient $\LTd\sslash\Waff$}\label{The Coarse Quotient Intro Subsection} Fix a split reductive group $G$ over a field $k$ of characteristic zero with choice of maximal torus $T$ and Borel subgroup $B \supseteq T$ and let $\LT$ denote the Lie algebra of $T$. As discussed above, this paper studies the coarse quotient $\LTd\sslash \Waff$, where $\Waff := \mathbb{Z}\Phi \rtimes W$ denotes the affine Weyl group for $G$, $W := N_G(T)/T$ is the (finite) Weyl group and $\mathbb{Z}\Phi$ denotes the root lattice of $G$. Here, we discuss desired properties of the coarse quotient $\LTd\sslash\Waff$ and give its formal definition. 

\subsubsection{Introduction to the Coarse Quotient $\LTd\sslash\Waff$}\label{Introduction to the Coarse Quotient of Affine Weyl Group Subsubsection}
The coarse quotient $\LTd\sslash \Waff$ is defined to serve as an analogue of the GIT quotient for the (in general infinite) group $\Waff$. However, one of the first obstructions which appears when working with the coarse quotient $\LTd\sslash \Waff$ is defining it so as to satisfy similar properties to the usual GIT quotients of finite groups. We survey some of these properties of the usual GIT quotient $V\sslash H$ of a Weyl group $H$ acting on some finite dimensional $k$-vector space $V$. Notice that, by definition of the GIT quotient for affine schemes, \[\mathcal{O}(V\sslash H) = \mathcal{O}(V)^H \xrightarrow{\sim} \text{lim}(\mathcal{O}(V) \rightrightarrows \mathcal{O}(V) \otimes_{\mathcal{O}(V)^H} \mathcal{O}(V))\] where rightmost maps are induced by the projections and the limit is taken in the category of ordinary $k$-algebras. When $H$ is a Weyl group, the $k$-algebra $\mathcal{O}(V) \otimes_{\mathcal{O}(V)^H} \mathcal{O}(V)$ admits a description as the \textit{union of graphs} $\Gamma_H$ of $H$ as a closed subscheme of $V \times V$, which we review in \cref{Union of Graphs is Product of LieTs over GIT Quotient}. In particular, we see that
\begin{equation}\label{Union of Graphs Is Product Over GIT Quotient}
    \text{the map }\Gamma_H \xrightarrow{} V \times_{V\sslash H} V\text{ is an equivalence}
\end{equation} and that, by definition, the affine scheme \begin{equation}\label{Quotient of Union of Finitely Many Graphs Acting on V} V\sslash H \text{ is a quotient of }V\text{ by the groupoid }\Gamma_H.\end{equation} Additionally, one of the key features of the GIT quotient $V\sslash H$ (which follows immediately from \cite[Amplification 1.3]{MumfordFogartyKirwanGeometricInvariantTheory}) is that \begin{equation}\label{The Quotient of the k-points is the k-points of the quotient}
    \text{the map } V(\overline{k})/H \xrightarrow{} (V\sslash H)(\overline{k}) \text{ is a bijection}
\end{equation} where $\overline{k}$ is the algebraic closure of $k$. 

Another useful feature of the quotient $V\sslash H$ which, for example, distinguishes it from the stack quotient $V/H$, is that the fiber of the quotient map depends on the choice of point of the codomain. Specifically, as we recall in \cref{Quotient map iso}, we have that
\begin{equation}\label{Fiber Product at Point Depends on Choice of Point}
    \text{for any }\lambda \in V(\overline{k})\text{, we have an isomorphism }\{\lambda\} \times_{V\sslash H} V \cong H \mathop{\times}\limits^{H_{\lambda}} \text{Spec}(C_{\lambda})
\end{equation} where $C_{\lambda}$ is a certain Artinian local $\overline{k}$-algebra determined by the action of the action of the stabilizer $H_{\lambda}$ of $\lambda$ on $V$. For example, by definition one has an equivalence
\raggedbottom
\[\{0\} \times_{V\sslash H} V \simeq \Spec(C_0)\]

\noindent where \[C_0 = \text{Sym}_{\overline{k}}(\overline{k} \otimes_k \Vdual)/\text{Sym}_{\overline{k}}(\overline{k} \otimes_k \Vdual)_+^H\] is the $(\overline{k}$-)\textit{coinvariant algebra}. On the other hand, if $H$ acts faithfully, then the action of $H$ on $V$ is generically free. In particular, for a generic $k$-point $\lambda$ of $V$, the orbit map gives an equivalence \begin{equation}\label{Generic Iso}\tag{*}\{\lambda\} \times_{V\sslash H} V \simeq \coprod_{h \in H} \Spec(k)\end{equation} as we review in \cref{We actually recover}. 

Finally, in representation theory, it is often of interest to study coherent sheaves on spaces such that $V\sslash H$ and relate them to equivariant coherent sheaves $\text{Coh}(V)^H$ on $V$. Specifically, if $q: V \to V\sslash H$ denotes the quotient map, then we have that \begin{equation}\label{Pullback Functor Is Fully Faithful With Easily Described Essential Image}
    q^*\text{ lifts to a fully faithful functor }\text{Coh}(V\sslash H) \xhookrightarrow{} \text{Coh}(V)^H\text{ with an easily described essential image}
\end{equation} as we review in much more detail in \cref{Descent to the Coarse Quotient for Pseudo-Reflection Groups Section}. 
\subsubsection{Other Potential Definitions of $\LTd\sslash\Waff$}
Using the above desiderata, we may immediately disqualify two potential definitions of $\LTd\sslash\Waff$:

\begin{enumerate}
    \item When $G = \SL_2$, we may identify $\LTd \cong \mathbb{A}^1 = \text{Spec}(k[x])$. In this case, since $\text{Spec}(k[x]^{\mathbb{Z}}) = \text{Spec}(k)$, the na\"ive definition for the coarse quotient $\text{Spec}(k[x]^{\Waff}) = \text{Spec}(k)$ does not satisfy \labelcref{Union of Graphs Is Product Over GIT Quotient}, \labelcref{The Quotient of the k-points is the k-points of the quotient}, or \labelcref{Fiber Product at Point Depends on Choice of Point}. 
    \item If the coarse quotient is replaced with the stack quotient $\LTd/\Waff$, \labelcref{Fiber Product at Point Depends on Choice of Point} is not satisfied.
\end{enumerate}

\subsubsection{The Coarse Quotient $\LTd\sslash\Waff$ Via Groupoids}
We will define $\LTd\sslash\Waff$ as a quotient of $\LTd$ by the union of graphs $\Gamma$ of $\Waff$. However, $\Gamma$ is not a scheme if $G$ is not a torus, and so a choice of ambient category is required to define this quotient analogous to \labelcref{Quotient of Union of Finitely Many Graphs Acting on V}. We choose what is essentially the most general setting and define $\LTd\sslash\Waff$ as a certain colimit in the category of \textit{prestacks} in the sense of \cite[Chapter 2.1]{GaRoI}. Specifically, let $\Gamma_{n}$ denote the $n$-fold product $\Gamma \times_{\LTd} \Gamma \times_{\LTd} ... \times_{\LTd} \Gamma$. Through the various source and target maps, the $\Gamma_{n}$ naturally form a groupoid object $\Gamma_{\bullet}$, see \cref{Groupoid Objects and Higher Algebra Subsection}. We now give a special case of our definition of the coarse quotient:

\begin{Definition}\label{Definition of Coarse Quotient for Coxeter Group}
The \textit{coarse quotient} $\LTd\sslash\Waff$ is the geometric realization of the simplicial object $\Gamma_{\bullet}$, in other words, \[\LTd\sslash\Waff := \text{colim}(\xymatrix{... \ar@<-2ex>[r] \ar[r] \ar@<2ex>[r] & \Gamma \ar@<-1ex>[r] \ar@<-1ex>[l] \ar@<1ex>[l] \ar@<1ex>[r]^{\text{t}} & \ar[l] \mathfrak{t}^{\ast}})\] in the category of prestacks.
\end{Definition}

As we review in \cref{Comparison to Classical GIT Construction Subsubsection}, replacing $\Gamma$ with $\Gamma_W$ in \cref{Definition of Coarse Quotient for Coxeter Group} gives the variety $\LTd\sslash W$ defined in \cref{Introduction to the Coarse Quotient of Affine Weyl Group Subsubsection} after sheafification. Observe that our \cref{Definition of Coarse Quotient for Coxeter Group} has the property that \labelcref{Quotient of Union of Finitely Many Graphs Acting on V} is satisfied vacuously. One advantage of taking \cref{Definition of Coarse Quotient for Coxeter Group} as our definition of $\LTd\sslash\Waff$ is that analogues of \labelcref{Union of Graphs Is Product Over GIT Quotient} and \labelcref{The Quotient of the k-points is the k-points of the quotient} hold essentially immediately, as we will see below. 

\subsubsection{Field-Valued Points}\label{Field-Valued Points Subsubsection} One disadvantage, however, of \cref{Definition of Coarse Quotient for Coxeter Group} is that the class of prestacks is \lq so general that it is, of course, impossible to prove anything non-trivial about\rq{} \cite{GaRoI} other than the formal properties discussed above. Therefore, much of the work in studying the coarse quotient will focus on showing the analogues of point \labelcref{Fiber Product at Point Depends on Choice of Point} and \labelcref{Pullback Functor Is Fully Faithful With Easily Described Essential Image} from the quotient map $\quotientmapforcoarsequotient: \LTd \to \LTd\sslash\Waff$. 

One issue, however, which immediately arises when attempting to study $\LTd\sslash \Waff$ in this way is that Zariski open subsets of $\LTd$ are \lq too large.\rq{} For example, when $G = \text{SL}_2$, we have
\raggedbottom
\[\Waff = 2\mathbb{Z} \rtimes \mathbb{Z}/2\mathbb{Z}\]

\noindent and, when considering the action of this group on $\mathbb{A}^1_k$, the $k$-points fixed by some order two subgroup of $\Waff$ behave differently than all other $k$-points, on which the action is free. Since the $k$-points fixed by some order two subgroup of $\Waff$ are given by a copy of the integers, we see that these points do not form a closed subscheme; they only form an \textit{ind-closed subscheme}. In particular, the subset of points for which the $\Waff$-action is free is not a Zariski open subset of $\mathbb{A}^1_k$.  

The main way we propose to work around this is to use \textit{field-valued points}, which can serve as a useful substitute when dealing with quotients such as $\LTd\sslash \Waff$. While in some algebro-geometric applications, generic points can often be more difficult to work with than general closed points, a general feature of representation theory is that $K$-points which do not factor through some $k$-point behave, informally speaking, as if they were generic $k$-points. For example, if $\ast$ denotes either a non-integer $k$-point of $\mathbb{A}^1_k$ or the generic point of $\mathbb{A}^1_k$, then there is a canonical isomorphism (see \cref{Soergel-Ish Isomorphism}):
\raggedbottom
\[\mathbb{A}^1/(2\mathbb{Z}) \times_{\mathbb{A}^1\sslash \Waff} \ast\textcolor{white}{l}\simeq\textcolor{white}{l}\ast \coprod \ast\]

\noindent whereas if $\ast$ denotes an integral $k$-point, one has a canonical isomorphism 
\raggedbottom
\[\mathbb{A}^1/(2\mathbb{Z}) \times_{\mathbb{A}^1\sslash \Waff} \ast\textcolor{white}{l}\simeq\textcolor{white}{l}\Spec(C_0)\]

\noindent where in this case $C_0 \cong k[\epsilon]/\epsilon^2$. A generalization of this plays an important role in representation theory, see \cref{Representation Theoretic Motivation Subsubsection}.

\subsection{Descent to the Coarse Quotient for Finite Groups}\label{Descent to the Coarse Quotient for Finite Pseudo-Reflection Groups Subsection} 
The tools used to study sheaves on the coarse quotient $\LTd\sslash \Waff$ also can be applied to obtain new results for quotients by reflection groups, and more generally \textit{pseudo-reflection} or \textit{complex reflection groups}. We survey these results here.

\subsubsection{Generalities on Descent to the Coarse Quotient}Let $H$ denote a finite group acting on an affine scheme $V$. When studying sheaves on $V\sslash H := \Spec(\O(V)^H)$, a key first insight is that any sheaf on $\vectorspaceHactson$ which is pulled back from $\vectorspaceHactson\sslash H$ canonically acquires an $H$-equivariant structure, or, equivalently, can be viewed as a sheaf on the \textit{stack quotient} $\vectorspaceHactson/H$. In the setting of quasicoherent sheaves, it is in fact not too difficult to show that the pullback 
\raggedbottom
\[\generalstacktoGITquotientmap^{\ast}: \QCoh(\vectorspaceHactson\sslash H) \to \QCoh(\vectorspaceHactson/H)\]
\noindent is fully faithful, as we review in \cref{Pullback From GIT to Stack Quotient of Finite Group of Affine Scheme is Fully Faithful}. 
\begin{Definition}
We say that $\F \in \QCoh(\vectorspaceHactson/H)$ \textit{descends to the coarse quotient} $V\sslash H$ if $\F$ is in the essential image of the pullback $\generalstacktoGITquotientmap^{\ast}$. 
\end{Definition}

In the course of studying sheaves on $\LTd\sslash \Waff$, we prove the following lemma, which says that the condition for a given $H$-equivariant sheaf to descend to the coarse quotient can be checked at each field-valued point of $V$. 

\begin{Lemma}\label{An Pseudo-Reflection Group Descends to Coarse Quotient Iff It Descends to Every Field-Valued Point}
A given $\F \in \QCoh(V)^H$ descends to the coarse quotient $V\sslash H$ if and only if for every field-valued point $x \in V(K)$, the induced $H_x$-representation on $x^*(\F)$ is trivial, where $H_x$ denotes the stabilizer of $x$. 
\end{Lemma}


\subsubsection{Descent to the Coarse Quotient for Pseudo-Reflection Groups} A case where one can profitably apply \cref{An Pseudo-Reflection Group Descends to Coarse Quotient Iff It Descends to Every Field-Valued Point} is the case where $H$ acts as a pseudo-reflection group:

\begin{Definition}\label{Pseudo-Reflection Group Definition}
We say that a finite group $H$ is a \textit{pseudo-reflection group} acting on some vector space $V$ if $H$ is generated by elements which act by \textit{pseudo-reflections} on $V$, which are in turn defined as non-identity endomorphisms\footnote{All finite order pseudo-reflections are diagonalizable over our characteristic zero field--see, for example, \cite[Section 14.3(a)]{KaneReflectionGroupsandInvariantTheory}. In particular, our definition agrees with others found in the literature.} of $V$ which fix a hyperplane of $V$ pointwise.
\end{Definition}

The groups acting by pseudo-reflections play a distinguished role in the theory of finite groups acting on vector spaces. For example, the Chevalley-Shephard-Todd theorem (which we review in \cref{Chevalley-Shephard-Todd Theorem}) says that the coarse quotient $V\sslash H$ is affine space if and \textit{only if} $H$ acts on $V$ as a pseudo-reflection group. Furthermore, a theorem of Steinberg \cite{SteinbergDifferentialEquationsInvariantUnderFiniteReflectionGroups} says that if $H$ is a pseudo-reflection group acting on a vector space, the stabilizer of any point is a pseudo-reflection group as well. Using this and \cref{An Pseudo-Reflection Group Descends to Coarse Quotient Iff It Descends to Every Field-Valued Point}, one can derive our first main result: 

\begin{Theorem}\label{Equivariant Sheaf for Pseudo-Reflection Group Descends iff It Descends for Each Reflection}
Assume $H$ is a pseudo-reflection group acting on some vector space $V$. A sheaf $\F \in \QCoh(\vectorspaceHactson)^H$ descends to the coarse quotient $\vectorspaceHactson\sslash H$ if and only if for every pseudo-reflection $r \in H$, the sheaf $\text{oblv}_{\langle r \rangle}^{H}(\F) \in \QCoh(\vectorspaceHactson)^{\langle r \rangle}$ descends to the coarse quotient $\vectorspaceHactson\sslash \langle r \rangle$. 
\end{Theorem}

Using \cref{Equivariant Sheaf for Pseudo-Reflection Group Descends iff It Descends for Each Reflection} and the fact that any reflection of a Coxeter group is conjugate to a simple reflection, one can derive an alternate proof of the main result of \cite{LoRemark}:

\begin{Theorem}\cite{LoRemark}
Assume $W$ is a Coxeter group with reflection representation $V$. A given $\F \in \QCoh(V)^W$ descends to the coarse quotient $V\sslash W$ if and only if for all simple reflections $s \in W$, $\text{oblv}^{H}_{\langle s \rangle}(\F) \in \QCoh(V)^{\langle s \rangle}$ descends to the coarse quotient $V\sslash\langle s \rangle$.
\end{Theorem}


\subsubsection{Local Descent to the Coarse Quotient}
\cref{An Pseudo-Reflection Group Descends to Coarse Quotient Iff It Descends to Every Field-Valued Point} is proved via study of the \textit{coinvariant algebra} $C$ for the action of $H$ on $V$, defined as the ring $\text{Sym}(\Vdual)/\text{Sym}(\Vdual)^H_+$, i.e. the quotient of $\text{Sym}(\Vdual)$ by the ideal generated by homogeneous polynomials of positive degree fixed by $H$. Since $H$ acts on $C$, we similarly obtain a stack quotient $\Spec(C)/H$ and a coarse quotient $\Spec(C)\sslash H$; however, the coarse quotient simplifies since 
\raggedbottom
\[\Spec(C)\sslash H := \Spec(C^H) \cong \Spec(k).\]

\noindent Similarly, we will recall that the pullback functor is fully faithful in \cref{Pullback From GIT to Stack Quotient of Finite Group of Affine Scheme is Fully Faithful} and therefore we can analogously define the notion of an object of $\QCoh(\Spec(C))^H$ \textit{descending to the coarse quotient}. A key technical tool in the proof of \cref{An Pseudo-Reflection Group Descends to Coarse Quotient Iff It Descends to Every Field-Valued Point} is the following proposition:

\begin{Proposition}\label{Equivariant Sheaf on Coinvariant Algebra Descends to the Coarse Quotient iff Fiber Is Trivial}
A given $\F \in \QCoh(\Spec(C))^H$ descends to the coarse quotient $\Spec(C)\sslash H = \Spec(k)$ if and only if the canonical $H$-representation on $i^{\ast}(\F)$ is trivial, where $i: \Spec(k) \xhookrightarrow{} \Spec(C)$ is the inclusion of the unique closed point. 
\end{Proposition}

\begin{Remark}
Our notation is inherently derived--see \cref{Conventions Subsection} for our exact conventions.
\end{Remark}

\subsection{Motivation}\label{Motivation Subsection}
The coarse quotient $\LTd\sslash \Waff$ can also be generalized to the coarse quotient of $\LTd$ by the action of the \textit{extended affine Weyl group} $\Wext := \characterlatticeforT \rtimes W$, where $\characterlatticeforT$ is the character lattice for $T$. This coarse quotient is expected to play an important role in geometric representation theory. In this section, we highlight some roles it plays and discuss some future applications.

\subsubsection{Representation Theoretic Motivation}\label{Representation Theoretic Motivation Subsubsection}
Assume, for the sake of exposition, that $G$ is semisimple and simply connected, and recall the \textit{BGG Category} $\mathcal{O}_{\lambda}$, a certain abelian category of representations with generalized central character $\chi_{\lambda}$. These categories and categories such as $\D(G)^{B \times B}$ (the category of bi $B$-equivariant $\D$-modules on $G$) can be understood through the use of Soergel modules and Soergel bimodules. We also recall that (ungraded) Soergel bimodules can be understood as certain quasicoherent sheaves on the union of graphs of the Weyl group $\Gamma_W$. As we review in \cref{Union of Graphs is Product of LieTs over GIT Quotient}, we can identify $\Gamma_W \simeq \LTd \times_{\LTd\sslash W} \LTd$. 

We can use the coarse quotient $\LTd\sslash \Wext$ to define one analogue of $\Gamma_W$ whose fibers, informally speaking, describe a subcategory of the BGG category $\mathcal{O}$ of the given central character. One can make this remark precise as follows. First, define:
\raggedbottom
\[\Gamma' := \LTd/\characterlatticeforT \times_{\LTd\sslash \Wext} \LTd/\characterlatticeforT\]

\noindent and let $\lambda \in \LTd(k)$ denote some $k$-point. We assume for the sake of exposition that $\lambda$ is an antidominant and regular with respect to the $(W, \cdot)$-action, and denote the associated image under the quotient map by $[\lambda]: \Spec(k) \to \LTd/\characterlatticeforT$. Then it is standard (see \cite[Chapter 4.9]{HumO}) that the blocks of $\mathcal{O}_{\lambda}$ are in bijection with the cosets $W/W_{[\lambda]}$, where $W_{[\lambda]}$ denotes the \textit{integral Weyl group}, see \cref{Integral Weyl Group Definition}. This group admits some remarkable properties--for example, $W_{[\lambda]}$ need not be a parabolic subgroup of the usual Weyl group and, moreover, as the notation suggests, this group only depends on $[\lambda]$. With this setup, we note that from \cref{SpecificGroupoidResults} and the Endomorphismensatz \cite{Soe1} one can derive the following:

\begin{Proposition}
We have $W$-equivariant isomorphisms
\raggedbottom
\[\Gamma' \times_{\LTd/\characterlatticeforT} \Spec(k) \simeq \LTd/\characterlatticeforT \times_{\LTd\sslash \Wext} \Spec(k) \simeq W \mathop{\times}\limits^{W_{[x]}} \text{Spec}(E_{\lambda})\]

\noindent where $E_{\lambda}$ is the endomorphism ring of the projective cover of the simple labelled by $\lambda$. 
\end{Proposition}

\subsubsection{Connections to Categorical Representation Theory}
The coarse quotient $\LTd\sslash \Wext$ has appeared recently in providing a coherent description of certain categories associated to the category of $\D$-modules on $G$, and more generally plays an important role in \textit{categorical representation theory}. For example, the coarse quotient implicitly appears in a recent result of Ginzburg and Lonergan: 

\begin{Theorem}\label{Abelian Categorical Mellin Transform for biWhittaker Sheaves} \cite{Lo}, \cite{Gin} There is an equivalence identifying the abelian category of bi-Whittaker $\D$-modules on $G$ with the full subcategory of $\Wext$-equivariant sheaves on $\LTd$ which descend to the coarse quotient in the sense of \cref{Descends To Coarse Quotient For Wext Definition}.
\end{Theorem}

It is argued in \cite{BZG} that the equivalence of \cref{Abelian Categorical Mellin Transform for biWhittaker Sheaves} can be upgraded to a monoidal equivalence conditional on a derived, mixed variant of the geometric Satake theorem, which in particular shows that the convolution monoidal structure on the derived category of bi-Whittaker $\D$-modules on $G$ is symmetric. In \cite{GannonCategoricalRepresentationTheoryandTheCoarseQuotient}, we prove this upgrade of \cref{Abelian Categorical Mellin Transform for biWhittaker Sheaves} unconditionally.

In \cite{BZG}, an argument is also sketched that there is a central functor from the category of sheaves on the coarse quotient $\LTd\sslash\Wext$ to the center of the monoidal category of $\D$-modules on $G$, i.e. the category $\mathcal{D}(G)^G$. It is also conjectured that this functor can be identified with a modified version of parabolic induction \cite[Conjecture 2.9]{BZG}. In \cite{GannonCategoricalRepresentationTheoryandTheCoarseQuotient}, we use the descent conditions outlined in this paper to construct a quotient of a candidate inverse to this functor. Moreover, this restricted parabolic induction functor can also be used to better understand character $\mathcal{D}$-modules on $G$, and was used to prove a conjecture of Braverman and Kazhdan on the acyclicity of $\rho$-Bessel sheaves on reductive groups, see \cite{ChenAVanishingConjecturetheGLnCase} and \cite{ChenOnTheConjecturesofBravermanKazhdan}. 

Moreover, in \cite{GannonCategoricalRepresentationTheoryandTheCoarseQuotient}, we argue that $\LTd\sslash \Wext$ plays an analogous role in the study of $\D(N\backslash G/N)$ as $\LTd\sslash W$ plays in the study of $\D(B\backslash G/B)$, see \cref{Representation Theoretic Motivation Subsubsection}. 

\begin{Theorem}\label{Main Monoidal Equivalence without T weak invariants}
\cite{GannonCategoricalRepresentationTheoryandTheCoarseQuotient} There is some quotient category of $\D(N\backslash G/N)$, denoted $\D(N\backslash G/N)_{\text{nondeg}}$, and a monoidal equivalence of categories
\raggedbottom
    \[\D(N\backslash G/N)_{\text{nondeg}} \simeq \IndCoh(\LTd/\characterlatticeforT \times_{\LTd\sslash\Wext} \LTd/\characterlatticeforT)\]
    \noindent which is $t$-exact up to cohomological shift. 
\end{Theorem}

Moreover, as we will show, the category $\D(N\backslash G/N)_{\text{nondeg}}$ can be described explicitly--for example, if $G = \text{SL}_2$, we can identify $\D(N\backslash G/N)_{\text{nondeg}}$ as the quotient of $\D(N\backslash G/N)$ by the full subcategory generated under colimits by the constant sheaf $\underline{k}_{N\backslash G/N}$. 

\subsection{Conventions}\label{Conventions Subsection}
We work over $k$, a field of characteristic zero, and by \textit{scheme}, we mean a $k$-scheme. If unspecified, $K/k$ denotes an arbitrary field extension. What follows is written in the language of derived algebraic geometry in the sense of \cite{GaRoI}, and, in particular, all categories of sheaves are written as \textit{DG categories}, or equivalently, $k$-linear presentable stable $\infty$-categories. However, in \cref{Descent to the Coarse Quotient for Pseudo-Reflection Groups Section} and \cref{The Union of Graphs of Affine Weyl Group is Ind Finite-Flat Section}, this usage is inessential--we only use this language to parallel the definitions of the coarse quotient $\LTd\sslash \Wext$ in \cref{The Coarse Quotient for Affine Weyl Group Section}. In particular, the reader only interested in results such as \cref{Equivariant Sheaf for Pseudo-Reflection Group Descends iff It Descends for Each Reflection} can read \cref{Descent to the Coarse Quotient for Pseudo-Reflection Groups Section} and \cref{The Union of Graphs of Affine Weyl Group is Ind Finite-Flat Section}, replacing our notion with the classical derived categories and derived functors between them, and lose no information. However, in \cref{The Coarse Quotient for Affine Weyl Group Section}, the usage of higher categories becomes more essential, as $\LTd\sslash \Wext$ is most naturally defined as a \textit{prestack}, see \cite[Chapter 2, Section 1]{GaRoI}. 

In particular, the classical abelian category of quasicoherent sheaves on a scheme $X$ (respectively, equivariant sheaves on a scheme $X$ with an $H$-action) on some scheme $X$ will be denoted $\QCoh(X)^{\heartsuit}$ (respectively, $\QCoh(X)^{H, \heartsuit}$). However, as we will review below in \cref{Chevalley-Shephard-Todd Theorem}, the Chevalley-Shephard-Todd Theorem gives that for any pseudo-reflection group $H$, the pullback functor $\phi^{\ast}$ is $t$-exact, and so it induces an exact functor of abelian categories
\raggedbottom
\[\QCoh(V\sslash H)^{\heartsuit} \xrightarrow{\phi^{\ast}} \QCoh(V)^{H, \heartsuit}\]
\noindent and the natural analogue of \cref{Equivariant Sheaf for Pseudo-Reflection Group Descends iff It Descends for Each Reflection} below holds for abelian categories. However, the abelian categorical analogues of results such as \cref{An Pseudo-Reflection Group Descends to Coarse Quotient Iff It Descends to Every Field-Valued Point} and \cref{Equivariant Sheaf on Coinvariant Algebra Descends to the Coarse Quotient iff Fiber Is Trivial} do \textit{not} hold, see \cref{Derived Fiber is Necessary for Pointwise Criterion}. The failure of these claims essentially stems from the failure of $i^{\ast}$ to be $t$-exact, or equivalently, the failure of the morphism $\text{Spec}(k) \xhookrightarrow{} \text{Spec}(C)$ to be flat. 

If $g$ is a point of a group acting on an affine scheme $X$, its \textit{graph} is the scheme-theoretic image of the closed embedding $X \xrightarrow{(g\cdot (-), \text{id})} X \times X$. The graph of $g$ is thus determined by an ideal $I_g$ of the ring of functions on $X \times X$. Given a finite subset $F$ of points of a group acting on an affine scheme, we define the \textit{union of graphs} $\Gamma_F$ as the closed subscheme cut out by the ideal $\cap_{g \in S} I_g$. Observe that, if $F \subseteq F'$ are finite subsets of points of a group acting on an affine scheme, then there is a natural closed embedding $\Gamma_F \xhookrightarrow{} \Gamma_{F'}$. We define the \textit{union of graphs} $\Gamma_S$ as the colimit over all $\Gamma_F$ for finite subsets $F \subseteq S$; thus $\Gamma_S$ is an ind-scheme but not necessarily a scheme.

\subsection{Acknowledgements}
I would like to thank David Ben-Zvi, Joakim F\ae geman, Rok Gregoric, Sam Gunningham, Allen Knutson, Gus Lonergan, Sam Raskin, Victor Reiner, Mark Shimozono, Brian Shin, Julianna Tymoczko, David Yang, and Alexander Yong for many interesting and useful conversations. I would also like to thank Sam Gunningham for some useful feedback on an earlier version of this draft, as well as the anonymous referees whose comments seriously improved the exposition of this paper. This project was largely completed while I was a graduate student at the University of Texas at Austin, and I would like to thank everyone there for contributing to such an excellent environment.

\section{Descent to the Coarse Quotient for Pseudo-Reflection Groups}\label{Descent to the Coarse Quotient for Pseudo-Reflection Groups Section}
\subsection{Preliminary Results} We recall that $k$ denotes a field of characteristic zero and, if unspecified, $K/k$ denotes an arbitrary field extension. 
\subsubsection{GIT Quotients}\label{GIT Quotients Subsubsection}
We temporarily assume $H$ is an arbitrary finite group acting on some arbitrary affine scheme $X := \Spec(A)$. Recall the notion of a \textit{GIT quotient} as in \cite{MumfordFogartyKirwanGeometricInvariantTheory}; we will only use the affine version so we content ourselves to define GIT quotient or \textit{coarse quotient} as $X\sslash H := \text{Spec}(A^H)$ \cite[Theorem 1.1]{MumfordFogartyKirwanGeometricInvariantTheory}. Direct computation shows the following result, a modification of which will be used to define the coarse quotient $\LTd\sslash \Wext$ in \cref{The Coarse Quotient for Affine Weyl Group Section}:

\begin{Proposition}\label{GIT Quotient is Colimit of Act Project}
If $H$ is a finite group acting on an affine scheme $X$ then the canonical map
\raggedbottom
\[X\sslash H \simeq \text{colim}(X \times_{X\sslash H} X \rightrightarrows X)\]

\noindent is an equivalence, where the colimit is taken in the (1,1)-category of classical affine schemes. 
\end{Proposition} 

\subsubsection{Stack Quotients}\label{Stack Quotients Subsubsection} Note that, in the notation of \cref{GIT Quotients Subsubsection}, the GIT quotient $X\sslash H$ is typically distinct from the stack quotient $X/H$ defined in, say, \cite[Chapter 2, Remark 4.3.8]{GaRoI}. We briefly recall this construction for the convenience of the reader. One can define a simplicial object in the category of ordinary affine schemes by the formula \[\Delta \ni [n] \mapsto H^{\times n} \times X\] for which the degeneracy map $d_0$ is given by projection away from the first factor and, if $i \in \{1, 2, ..., n\}$ and $j \in \{0, 1, ..., n-1\}$ then \begin{equation}\label{Groupoid Formulas}d_{i}(h_n, ..., h_0) = (h_{n}, h_{n-1}, ..., h_{n - i + 1}h_{n-i}, ..., h_0), \text{ and } s_j(h_n, ..., h_0) = (h_n, ..., h_{n-i + 1}, 1, h_{n - i}, ..., h_0)\end{equation} for $h_0 \in X$ and $h_{\ell} \in H$ for $\ell \geq 1$. From the inclusion of ordinary affine schemes into PreStk given by the Yoneda embedding, one obtains a simplicial object in the category PreStk. Define $X/H$ as the \'etale sheafification of the geometric realization of this simplicial object in PreStk.

It is standard (and an immediate consequence of \cref{GroupoidResults} below and the fact that we sheafification commutes with finite limits \cite[Lemma 2.3.6]{GaRoI}) that we have an equivalence \begin{equation}\label{Act Project Equivalence}(\mathrm{act}, \mathrm{proj}): H \times X \xrightarrow{\sim} X \times_{X/H} X\end{equation} induced by the action map $\mathrm{act}$ and the projection map $\mathrm{proj}$. As in the proof of \cite[Chapter 2, Proposition 4.3.6]{GaRoI}, $X \to X/H$ is a flat covering, and so we have induced equivalences \begin{equation}\label{Induced Equivalence by Flat Covering}\QCoh(X/H) \xrightarrow{\sim} \lim_{\Delta}\QCoh(H^{\bullet} \times X) \xleftarrow{\sim} \lim_{\Delta}\QCoh(H)^{\otimes \bullet} \otimes \QCoh(X) =: \QCoh(X)^{H, w}\end{equation} given by pullback (since $\QCoh$ satisfies flat descent \cite[Theorem 1.3.4]{GaRoI}) and the symmetric monoidality of the functor \lq -mod\rq{} \cite[Chapter 1, Sect. 8.5.10]{GaRoI}. We will implicitly use the composite equivalence \begin{equation}\label{QCoh of Stack Quotient is Equivariant Sheaf Category}\QCoh(X/H) \xrightarrow{\sim} \QCoh(X)^{H, w}\end{equation} in what follows. 

The following result informally states that the category $\QCoh(X/H)$ can be identified with the category whose objects are objects of $\QCoh(X)$ equipped with an $H$-representation and various higher compatibilities; it is well known, but we were unable to locate a reference for it so we recall the proof for the convenience of the reader: 

\begin{Proposition}\label{QCoh on Stack Quotient}
If $H$ is a finite group (or any affine algebraic group) acting on an affine scheme $X$, there is an equivalence \[\QCoh(X/H) \xrightarrow{\sim} \O(H \times X)\mathrm{-comod}(\QCoh(X))\] induced by pullback by the quotient map $f: X \to X/H$.
\end{Proposition}

\begin{proof}
Observe that $(-)^{\mathrm{op}}: \mathrm{DGCat} \to \mathrm{DGCat}$ is an autoequivalence of categories such that $(R^{\mathrm{op}}, L^{\mathrm{op}})$ is an adjoint pair if $(L, R)$ is. In particular, we obtain an equivalence \begin{equation}\label{Opposite Iso}\QCoh(X/H)^{\mathrm{op}} \xrightarrow{\sim} \lim_{\Delta} \QCoh(H^{\bullet} \times X)^{\mathrm{op}}\end{equation} induced by the various pullback maps. Moreover, the functor $\mathrm{act}_*\mathrm{proj}^*$ naturally lifts to the structure of a \textit{monad} acting on $\QCoh(X)^{\mathrm{op}}$. We next claim that there is an equivalence \begin{equation}\QCoh(H^{\bullet} \times X)^{\mathrm{op}} \xrightarrow{\sim} \mathrm{act}_*\mathrm{proj}^*\mathrm{-mod}(\QCoh(X)^{\mathrm{op}})\end{equation} induced by $f^*$. To this end, we first observe that, for every $\zeta: [m] \to [n]$ in $\Delta$, the diagram \begin{equation}\label{Left Adjointable Diagram}
  \xymatrix@R+2em@C+2em{
\QCoh(H^{n} \times X)^{\mathrm{op}} \ar[r]^{d_0^*} \ar[d]^{\varepsilon_{n}^*} &  \QCoh(H^{n + 1} \times X)^{\mathrm{op}} \ar[d]^{\varepsilon_{n + 1}^*}\\
 \QCoh(H^{m} \times X)^{\mathrm{op}} \ar[r]^{d_0^*}  &   \QCoh(H^{m + 1} \times X)^{\mathrm{op}}
  }
\end{equation} commutes by the functoriality of $\QCoh^*$, where $\varepsilon_n$ is the map of affine schemes associated to $\zeta$ and $\varepsilon_{n + 1}$ is the map associated to the unique extension of $\zeta$ sending $m + 1$ to $n + 1$. Moreover, the diagram \labelcref{Left Adjointable Diagram} is left adjointable in the sense of \cite[Definition 7.3.1.2]{LuHTT} by say \cite[Proposition 2.2.2]{GaRoI}. Therefore, \cite[Theorem 4.7.5.2(3)]{LuHA} gives our desired equivalence. Taking opposites, we obtain an equivalence  \begin{equation}\QCoh(H^{\bullet} \times X) \xrightarrow{\sim} \mathrm{act}_*\mathrm{proj}^*\mathrm{-comod}(\QCoh(X))\end{equation} induced by $f^*$. 
Finally, since $\QCoh(X) := A\mathrm{-mod}$, its dual can be naturally identified with $A^{\mathrm{op}}\mathrm{-mod}$ ($= A\mathrm{-mod}$) and there is a monoidal equivalence \[(\uEnd_{\mathrm{DGCat}}(\QCoh(X)), \circ) \xrightarrow{\sim} (A \otimes A^{\mathrm{op}}\mathrm{-mod}, - \otimes_A -)\] compatible with the natural actions on $\QCoh(X) := A\mathrm{-mod}$. Under this equivalence, the comonad $\mathrm{act}_*\mathrm{proj}^*$ naturally corresponds to the coalgebra object $\O(H \times X)$, and so our claim follows.
\end{proof}
\subsubsection{GIT Quotients and Stack Quotients}Let $\generalstacktoGITquotientmap: X/H \to X\sslash H$ denote the quotient map.\footnote{More specifically: the universal property of this colimit gives a map from the non-sheafified analogue of $X/H$, and the map $\phi$ is induced from this map and the fact that the affine scheme $X\sslash H$ is a sheaf \cite[Lemma 2.2.2.13]{ToenVezzosiHomotopicalAlgebraicGeometryIIGeometricStacksandApplications}, \cite[Proposition 2.4.2]{GaRoI}.} We now recall the following result: 

\begin{Proposition}\label{Pullback From GIT to Stack Quotient of Finite Group of Affine Scheme is Fully Faithful}
 The pullback $\generalstacktoGITquotientmap^{\ast}$ is fully faithful.
\end{Proposition}

\begin{proof}
Recall $A := \O(X)$; let $C := \O(H \otimes X)$, $\C := C\mathrm{-comod}(\QCoh(X))$, and $E: \QCoh(X/H) \xrightarrow{\sim} \C$ denote the equivalence of \cref{QCoh on Stack Quotient}. To prove that $\generalstacktoGITquotientmap^*$ is fully faithful, it suffices to prove that the functor $E \circ \generalstacktoGITquotientmap^*$ is fully faithful. Observe that this functor can be obtained from tensoring with the monoidal unit $S := \O(V)$. Since \[\uHom_{\C}(S, M) \xleftarrow{\sim} \uHom_\C(S, \mathrm{colim}(\hspace{-0.04in}\xymatrix{C \otimes_{A} M \ar@<-1ex>[r] \ar@<1ex>[r] & C \otimes_{A} C \otimes_A M \ar[l] \ar@<-2ex>[r] \ar[r] \ar@<2ex>[r] & \hdots \ar@<-1ex>[l] \ar@<1ex>[l]
}))\] \[\xrightarrow{\sim} \lim(\hspace{-0.03in}\xymatrix{\uHom_{\C}(S, C \otimes_{A} M) \ar@<-1ex>[r] \ar@<1ex>[r] & \uHom_{\C}(S, C \otimes_A C \otimes_{A} M) \ar[l] \ar@<-2ex>[r] \ar[r] \ar@<2ex>[r] & \hdots \ar@<-1ex>[l] \ar@<1ex>[l]
})\]
\[\simeq \lim(\hspace{-0.03in}\xymatrix{\uHom_{S\mathrm{-mod}}(S, M) \ar@<-1ex>[r] \ar@<1ex>[r] & \uHom_{S\mathrm{-mod}}(S, C \otimes_A M)  \ar[l] \ar@<-2ex>[r] \ar[r] \ar@<2ex>[r] & \hdots \ar@<-1ex>[l] \ar@<1ex>[l]
})\]

\[\simeq \lim(\hspace{-0.04in}\xymatrix{M \ar@<-1ex>[r] \ar@<1ex>[r] & C \otimes_A M  \ar[l] \ar@<-2ex>[r] \ar[r] \ar@<2ex>[r] & \hdots \ar@<-1ex>[l] \ar@<1ex>[l]
}) \simeq M^H\] as vector spaces, we deduce that the right adjoint to this functor is given by the invariants functor. Since the invariants functor is continuous\footnote{Indeed, since $\mathrm{Rep}(H)$ is the derived category of its heart \cite[Lemma 6.4.6]{DrinfeldGaitsgoryOnSomeFinitenessQuestionsforAlgebraicStacks} the characteristic zero assumption on our field implies, by the Barr-Beck-Lurie theorem, that $\mathrm{Rep}(H)$ is equivalent to the product of the categories $\uEnd(V)$-mod where $V$ varies over isomorphism classes of irreducible representations of representations of $H$. The trivial functor is right \textit{and} left adjoint to the projection onto $\mathrm{Vect} = \uEnd(k)$-mod, and so it commutes with all colimits.} we have that, for any $M \in \QCoh(V\sslash H)$ that the rightmost map in the composite \[M \xrightarrow{\sim} \O(V)^H \otimes_{\O(V)^H} M \xrightarrow{} (\O(V) \otimes_{\O(V)^H} M)^H\] is an equivalence. Since this composite is the unit map for $E \circ \generalstacktoGITquotientmap^*$, we deduce that $E \circ \generalstacktoGITquotientmap^*$ is an equivalence, as desired.
\end{proof}

\subsubsection{Results on Pseudo-Reflection Groups} We now recall a key result in the theory of finite group actions on vector spaces:

\begin{Theorem}\label{Chevalley-Shephard-Todd Theorem}(Chevalley-Shephard-Todd Theorem \cite[Chapter 5, Theorem 4]{BourbakiGroupesEtAlgebresdeLieChaptersIVVVI}, \cite{ChevalleyInvariantsofFiniteGroupsGeneratedByReflections}, \cite{ShephardToddFiniteUnitaryReflectionGroups})
If $H$ is any finite group acting on a vector space of dimension $d$, then $V\sslash H \cong \mathbb{A}_k^d$ if and only if $H$ is a pseudo-reflection group. 
\end{Theorem}

 \begin{Remark}
%
More generally, in fact, the space $\vectorspaceHactson\sslash H$ is smooth only if $H$ is a pseudo-reflection group. Since we do not use this theorem in any substantial way, we only remark that this theorem is a consequence of \cref{Chevalley-Shephard-Todd Theorem} and the purity of the branch locus theorem. 
\end{Remark}

We were unable to locate a proof of the following standard result for general pseudo-reflection groups\footnote{For a proof in the reflection group case, see \cite[Section 3.5]{HumRe}.}, so we recall it here:

\begin{Proposition}\label{For Pseudo-Reflection Groups Acting Faithfully on Vector Space The GIT Quotient is Flat}
If $H$ is a pseudo-reflection group acting faithfully on some vector space $V$, then $\text{Sym}(\Vdual)$ is a free $\text{Sym}(\Vdual)^H$-module of rank $|H|$. 
\end{Proposition}

\begin{proof}
The fact that $\text{Sym}(\Vdual)$ is a free $\text{Sym}(\Vdual)^H$-module of some finite rank is standard, see, for example, \cite[Chapter 18-3]{KaneReflectionGroupsandInvariantTheory}. To show that this rank is $|H|$, note that since $H$ acts faithfully and $V$ is irreducible, there exists some $H$-invariant nonempty open subset $U \subseteq V$ on which the action is free. Since the quotient map $q: V \to V\sslash H$ is a flat morphism of finite type Noetherian schemes, the set $q(U)$ is open. Note also that because $q$ is a uniform categorical quotient \cite[Chapter 1, \textsection 2, Theorem 1.1]{MumfordFogartyKirwanGeometricInvariantTheory}, the map $q|_U: U \to q(U)$ is a categorical quotient. 

Let $j: q(U) \xhookrightarrow{} V\sslash H$ denote the open embedding. It now suffices to show that the rank of $j^{\ast}(\text{Sym}(\Vdual))$ is $|H|$, and we may show this, in turn, by computing the rank of the free module $q|_{U}^{\ast}j^{\ast}(\text{Sym}(\Vdual))$ However, note that since the action of $H$ on $U$ is free, the leftmost of the following diagrams are Cartesian by \cite[Chapter 0, \textsection 4, Proposition 0.9]{MumfordFogartyKirwanGeometricInvariantTheory}:
\begin{equation}\label{Cartesian Diagram Exhibiting Properties from Faithful Action is Free on Dense Open}
  \xymatrix@R+2em@C+2em{
   H \times U \ar[r]^{\text{act}} \ar[d]^{\text{proj}}& U \ar[d]^{q|_{U}} \ar[r] & V \ar[d]^{q}\\
  U \ar[r]^{q|_{U}} & q(U) \ar[r]^{j} & V\sslash H
  }
 \end{equation}
 
 \noindent and rightmost box is Cartesian by construction. We then see that we may compute, by base change along the \lq large\rq{} Cartesian diagram of \labelcref{Cartesian Diagram Exhibiting Properties from Faithful Action is Free on Dense Open} to see that the rank of the free module $q|_{U}^{\ast}j^{\ast}(\text{Sym}(\Vdual))$ is $|H|$ as desired. 
\end{proof}

We also have the following theorem of Steinberg, stated in terms of our conventions on pseudo-reflection groups: 
\begin{Theorem}\label{Stabilizer of Element Under Pseudo-Reflection Group is Pseudo-Reflection Group}\cite[Theorem 1.5]{SteinbergDifferentialEquationsInvariantUnderFiniteReflectionGroups} 
The stabilizer $H_x$ of some $x \in V$ is a pseudo-reflection group. 
\end{Theorem} 

\subsubsection{The Coinvariant Algebra}
Assume $H$ is some finite group acting on a vector space $V$, and let $C_0$ denote the \textit{coinvariant algebra} of this action, i.e.
\raggedbottom
\[C_0 := \text{Sym}(\Vdual) \otimes_{\text{Sym}(\Vdual)^H} k\]
\noindent where the ring map $\text{Sym}(\Vdual)^H \to k$ is given by the ring map sending all elements of positive degree to zero. More generally, if $x \in \vectorspaceHactson(K)$ is some field-valued point, we define the algebra $C_x$ as the $K$-algebra
\raggedbottom
\[C_x := \text{Sym}(\Vdual_K) \otimes_{\text{Sym}(\Vdual_K)^{H_x}} K\]

\noindent where the ring map $\text{Sym}(\Vdual_K)^{H_x} \to K$ is given by evaluation at $x$. Note that the map $\tau_x: V_K \to V_K$ given by translation by $x$ gives an isomorphism between $C_x$ and the coinvariant algebra for the action of $H_x$ on $V_K$. 

\subsubsection{Nonzero Sheaves on Smooth Schemes Have Nonzero Fiber}
We now recall the following result, which says that a given (complex of) quasicoherent sheaves is nonzero only if some fiber at some field-valued point is nonzero.

\begin{Proposition}\label{For any qcoh sheaf on smooth classical Noetherian scheme there exists a field valued point where the fiber doesn't vanish}
Assume $X$ is a Noetherian classical scheme and $\F \in \QCoh(X)$ is nonzero. Then there exists some field-valued point $x: \Spec(K) \to X$ such that $x^*(\F)$ is nonzero.
\end{Proposition}

The following proof is due to Arinkin \cite[Lemma 10]{Ari}. We recall it here for the convenience of the reader, and to show that the proof does not require the assumption that the associated complex of $\F$ be bounded. 

\begin{proof}
Since locally Noetherian schemes are covered by $\text{Spec}(A)$ for Noetherian $A$, it suffices to prove this when $X = \text{Spec}(A)$. Now, let $M$ denote some $A$-module whose fiber vanishes at every field-valued point, and consider all closed subschemes $i:Z \xhookrightarrow{} X$ such that $i^{\ast}(M)$ is nonzero. Since $X$ is Noetherian, there is a minimal closed subscheme with this property. Taking the fiber of $M$ at this closed subscheme, we may replace this minimal closed subscheme by $X$ itself. In particular, for any $f \in A$ which is nonzero, since we have a fiber sequence $A \xrightarrow{f\cdot -}  A \to A/f$, we have that multiplication by any nonzero $f \in A$ yields an equivalence $f\cdot -: M \xrightarrow{\sim} M$. 

If $A$ is not an integral domain, then there exists some nonzero $f, g \in A$ whose product is zero. We therefore see that the zero map is an isomorphism since the composite $M \xrightarrow{f \cdot -} M \xrightarrow{g \cdot -} M$ is an isomorphism. Now assume $A$ is an integral domain, and fix some $i \in \mathbb{Z}$. The above analysis gives any nonzero $f \in A$ acts invertibly on $H^i(M)$ for all $i$. Therefore, we have that, as classical $A$-modules, $H^i(M) \cong K \otimes_A H^i(M)$, where $K$ denotes the field of fractions of $A$. However, since localization is an exact functor of abelian categories, the functor $K \otimes_A -$ is a $t$-exact functor of derived categories. Thus, we have 
\raggedbottom
\[0 = H^i(K \otimes_A M) \cong K \otimes_A H^i(M) \cong H^i(M)\] where the first step uses the assumption that the fiber of $M$ at every point vanishes. Since $H^i(M) = 0$ for all $i$, we have that $M$ itself vanishes in $A$-mod by the left and right completeness of the $t$-structure on $A$-mod. 
\end{proof}

\subsection{Pointwise Descent to the Coarse Quotient}
Assume $H$ is a finite group which acts on some vector space $V$, and let $C$ denote the coinvariant algebra for this action. The main result of this section is the proof of \cref{Equivariant Sheaf on Coinvariant Algebra Descends to the Coarse Quotient iff Fiber Is Trivial}, which provides a \lq pointwise\rq{} criterion for an $H$-equivariant sheaf in $\QCoh(\text{Spec}(C)/H)$ to descend to the coarse quotient $\text{Spec}(C)\sslash H = \Spec(k)$. 

\begin{Proposition} Let $\pointwisegeneralstacktoGITquotientmap: \text{Spec}(C)/H \to \text{Spec}(k)$ denote the terminal map of $k$-schemes. 
\begin{enumerate}
    \item The functor $\pointwisegeneralstacktoGITquotientmap^{\ast}$ is $t$-exact and fully faithful. 
    \item Under the equivalence $\QCoh(\text{Spec}(C)/H) \simeq \QCoh(\text{Spec}(C))^H$, the essential image of $\pointwisegeneralstacktoGITquotientmap^{\ast}$ is the full subcategory generated by object $C \in \QCoh(\text{Spec}(C))^H$ with its canonical equivariance. 
\end{enumerate}
\end{Proposition}

\begin{proof}
The $t$-exactness follows since any functor $F: \text{Vect} \to \C$ to some DG category $\C$ equipped with a $t$-structure which sends the one dimensional vector space $k \in \text{Vect}^{\heartsuit}$ to an object in $\C^{\heartsuit}$ is $t$-exact, and the fully faithfulness follows from \cref{Pullback From GIT to Stack Quotient of Finite Group of Affine Scheme is Fully Faithful}. Moreover, since the category $\text{Vect}$ is generated under colimits by $k$, the essential image of the fully faithful functor $\pointwisegeneralstacktoGITquotientmap^{\ast}$ is generated under colimits by the essential image of $\pointwisegeneralstacktoGITquotientmap^{\ast}(k)$. Since this object is given by $C \in \QCoh(\text{Spec}(C))^H$ with its canonical equivariance, we obtain (2).
\end{proof}

Let $i: \Spec(k) \xhookrightarrow{} \Spec(C)$ denote the embedding of the unique closed point. To prove \cref{Equivariant Sheaf on Coinvariant Algebra Descends to the Coarse Quotient iff Fiber Is Trivial}, we first show the following lemma: 


\begin{Lemma}\label{If the fiber of a C-module is trivial then the fixed points are the fixed points of the fiber}
Assume $M \in C\text{-mod}^{H}$ has the property that the (derived) fiber $i^*(M) \simeq k \otimes_C M \in \text{Rep}_{k}(H)$ lies in the full subcategory generated by the trivial $H$-representation. Then the unit map $M \xrightarrow{} k \otimes_C M$ induces an equivalence $M^H \xrightarrow{\sim} (k \otimes_C M)^H = k \otimes_C M$, where we view $M$ and $k \otimes_C M \simeq i_*i^*(M)$ as objects of $\text{Rep}(H)$ via the composite $C\text{-mod}^H \xrightarrow{\terminalmapfromC_*} \text{Vect}^H \simeq \text{Rep}(H)$. 
\end{Lemma}

\begin{proof}
Consider the cofiber sequence 
\raggedbottom
\[C^+ \to C \to k\]

\noindent induced by the short exact sequence of classical $C$-modules equipped with $H$-equivariance. Upon taking the (derived) tensor product with $M$, we obtain a cofiber sequence
\raggedbottom
\begin{equation}\label{Cofiber Sequence of Coinvariant Algebra Tensored with M}
C^+ \otimes_C M \to M \to k \otimes_C M
\end{equation}

\noindent of objects of $C\text{-mod}^H$. Therefore, it suffices to show that if $M$ has the property that the derived fiber $i^*(M)$ is a complex of trivial $H$-representations, then $(C^+ \otimes_C M)^H \simeq 0$. 

Note that $C$ admits a filtration $C_0 \subseteq C_1 \subseteq ...  \subseteq C_{\ell}$ induced by the degree of $\text{Sym}(\Vdual)$, and the $H$-action preserves this filtration. In particular, we may filter $C^+$ by classical $C$-modules for which $C^+$ acts trivially on the associated graded. Furthermore, since $C^+$ has no trivial subrepresentations (by definition of the coinvariant algebra), for each of these subquotients $S$ in the filtration of $C^+$, we see that 
\raggedbottom
\[(S \otimes_C M)^H \simeq ((S \otimes_{k} C/C^+) \otimes_C M)^H \simeq (S \otimes_{k} (k \otimes_C M))^H\]
\noindent is the tensor product of some nontrivial $H$-representation $S$ over $k$ with an entirely trivial representation (by assumption on $M$). Therefore we see that $(S \otimes_C M)^H \simeq 0$, and so $(C^+ \otimes_C M)^H \simeq 0$, as required. 
\end{proof}

\begin{proof}[Proof of \cref{Equivariant Sheaf on Coinvariant Algebra Descends to the Coarse Quotient iff Fiber Is Trivial}]
Because the pullback map $(\text{Spec}(k)/H \to \text{Spec}(k))^{\ast}$ corresponds to the inclusion of the trivial $H$-representation under the equivalence $\QCoh(\text{Spec}(k)/H) \simeq \text{Rep}(H)$, we see that all objects in the essential image of $\pointwisegeneralstacktoGITquotientmap^*$ have the property that the $H$-representation of the (derived) fiber at the unique closed point is trivial. Moreover, the functor $\pointwisegeneralstacktoGITquotientmap^*$ admits a right adjoint $\pointwisegeneralstacktoGITquotientmap_*$ which is, at the level of homotopy categories, explicitly given by taking the $H$-fixed points of the underlying complex of $H$-representations. Therefore, it suffices to show that $\pointwisegeneralstacktoGITquotientmap_*$ is conservative on those objects of $\QCoh(\Spec(C))^H$ for which the canonical $H$-representation on the fiber at the closed point is trivial. Let $\F$ be a nonzero object in this subcategory. Since $\Spec(C)$ has a unique closed point, we see that $i^{\ast}(\F)$ is nonzero. Therefore, by \cref{If the fiber of a C-module is trivial then the fixed points are the fixed points of the fiber}, we see that the $H$-fixed points of $\F$ itself is nonzero, as desired. 
\end{proof}

\begin{Example}\label{Derived Fiber is Necessary for Pointwise Criterion}
This example shows  taking derived fiber is necessary in \cref{Equivariant Sheaf on Coinvariant Algebra Descends to the Coarse Quotient iff Fiber Is Trivial}. Let $H$ denote the order two group, and consider its action on a one dimensional vector space $V$ given by scaling by $-1$. The associated coinvariant algebra is $A := k[x]/x^2$. Consider the $A$-module $k \cong A/x$ with its canonical $H$-equivariance. This object is not in the essential image of $\pointwisegeneralstacktoGITquotientmap^*$ since any object mapping to it must lie in the heart of $\text{Vect}$ by the $t$-exactness of $\pointwisegeneralstacktoGITquotientmap^*$, but the underlying $C$-modules of all objects in the essential image of $\pointwisegeneralstacktoGITquotientmap^*|_{\text{Vect}^{\heartsuit}}$ have even $k$-dimension. 
One can explicitly compute, however, that the $H$-representation on $H^{-1}i^{\ast}(k)$ is given by the sign representation. Alternatively, using the short exact sequence of $H$-equivariant $A$-modules (all in the ordinary abelian category)
\[0 \to k_{\text{sign}} \xhookrightarrow{1 \mapsto x} A \to k \to 0\]

\noindent where here $k_{\text{sign}}$ is $k$ as an $A$-module and equipped with the sign equivariance, one obtains an isomorphism $H^{-1}i^{\ast}(k \otimes_{A} k) \xrightarrow{\sim} H^{0}i^{\ast}(k_{\text{sign}} \otimes_{A} k)$, which is manifestly nontrivial as an $H$-representation. 
\end{Example}

\newcommand{\generalVectorSpacetoGITQuotientMap}{\tilde{\phi}}
\subsection{Global Descent to the Coarse Quotient}\label{Global Descent to the Coarse Quotient Subsection}
Assume that $H$ is a pseudo-reflection group acting by reflections on some finite dimensional $k$-vector space $V$. By \cref{Pullback From GIT to Stack Quotient of Finite Group of Affine Scheme is Fully Faithful}, the pullback by the quotient map $\generalstacktoGITquotientmap: V/H \to V\sslash H$ is fully faithful. Our goal in \cref{Global Descent to the Coarse Quotient Subsection} will be to give an explicit description of the essential image of this functor in terms of field-valued points of $V$. 

To do this, we first set the following notation: for any fixed field-valued point $x: \Spec(K) \to V$ and any group $H'$ acting on $V$, we let $V \times_{V\sslash H'} \Spec(K)$ denote the fiber product of the diagram \begin{equation*}
  \xymatrix@R+2em@C+2em{
 V \times_{V\sslash H'} \Spec(K) \ar[d] \ar[r] & \Spec(K) \ar[d]_{\Phi_{H'} \circ x} \\
 V \ar[r]^{\Phi_{H'}}   &  V\sslash H'
  } 
\end{equation*} where $\Phi_{H'}: V \to V\sslash H'$ is the quotient map. Letting $X_K := X \times_{\Spec(k)} \Spec(K)$ for any $k$-scheme $X$, we will use similar notation for the product $V_K \times_{V_K\sslash H'} \Spec(K)$. Observe that we have equivalences \[\Spec(C_x) \xrightarrow{\sim} V_K \times_{V_K\sslash H_x} \Spec(K) \xrightarrow{\sim} V \times_{V\sslash H_x} \Spec(K)\] given by the translation map and the projection map. We let $\tilde{x}: \Spec(C_x) \to V$ denote the composite of these equivalences and the projection map onto $V$, and let $\dot{x}: \Spec(C_x)/H_x \to V/H$ denote the map induced by the composite of $\tilde{x}$ and the quotient map $V \to V/H$. 

\begin{Theorem}\label{Fully Faithfulness and Easy Essential Image of Pullback by generalstacktoGITquotientmap Revamp} The following are equivalent for a given $\F \in \QCoh(V)^{H}$:
\begin{enumerate}
    \item The complex $\F$ descends to the quotient $V\sslash H$. 
    \item The complex $\tilde{x}^*(\F) \in \QCoh(\Spec(C_x))^{H_x}$ descends to the coarse quotient $\text{Spec}(C_x)\sslash H_x = \text{Spec}(K)$ for every field-valued point $x$ of $V$. 
    \item The $H_x$-representation on $x^{\ast}(\text{oblv}^{H}_{H_x}(\F)) \in \QCoh(\Spec(K))^{H_x} \simeq \mathrm{Rep}_K(H_x)$ is trivial for every field-valued point $x \in V(K)$. 
\end{enumerate}
\end{Theorem}

We prove \cref{Fully Faithfulness and Easy Essential Image of Pullback by generalstacktoGITquotientmap Revamp} after proving a lemma and deriving a corollary from it. 

\begin{Lemma}\label{Quotient map iso}
For a field-valued point $x: \Spec(K) \to V$, the closed embedding \[\Spec(C_x) \xrightarrow{\sim} V \times_{V\sslash H_x}\Spec(K) \xhookrightarrow{} V \times_{V\sslash H} \Spec(K)\] induces an isomorphism \begin{equation}\label{Action Map Iso}H \times^{H_x} \Spec(C_x) \xrightarrow{\sim}V \times_{V\sslash H} \Spec(K)\end{equation} of affine schemes via the action map.
\end{Lemma}

\begin{proof}
Since the map $H \times^{H_x} \Spec(C_x) \to H/H_x$ is an affine map onto an affine scheme,\footnote{In general, the quotient of a discrete group $H'$ by a subgroup $H'_0$ is an affine scheme: indeed, if $Q := \Spec(\O(H')^{H'_0})$  then one can directly construct an isomorphism $H_0' \times H \xrightarrow{\sim} H \times_Q H$ of affine schemes and, similarly, an isomorphism of simplicial objects $H_0'^{\bullet} \times H$ and the \^Cech nerve of $H \to Q$. Since any groupoid object in an $\infty$-topos such as PreStk is effective \cite[Theorem 6.1.0.6]{LuHTT}, we deduce that $H'/H'_0 \xrightarrow{\sim} Q$ as desired.} $H \times^{H_x} \Spec(C_x)$ is an affine scheme. Since the induced map \labelcref{Action Map Iso} commutes with base change in the natural way, we may assume that $K = k$ is algebraically closed. Observe that, in this case, by \cite[Amplification 1.3]{MumfordFogartyKirwanGeometricInvariantTheory}, the $k$-points in the fiber $V \times_{V\sslash H} \Spec(k)$ are in bijective correspondence with the $k$-points of the $H$-orbit of $x$. Therefore the map \labelcref{Action Map Iso} is a closed embedding, since this can be checked Zariski locally. By \cref{For Pseudo-Reflection Groups Acting Faithfully on Vector Space The GIT Quotient is Flat}, the $k$-dimension of the ring of functions on $V \times_{V\sslash H} \Spec(k)$ is precisely $|H|$. Similarly, since $H_x$ is a reflection group by \cref{Stabilizer of Element Under Pseudo-Reflection Group is Pseudo-Reflection Group}, the $k$-dimension of the ring of functions on $V \times_{V\sslash H_x} \Spec(k)$ is precisely $|H_x|$, and so the $k$-dimension of the ring of functions on $H \times^{H_x} \Spec(C_x)$ is exactly $[H: H_x]|H_x| = |H|$. Since \labelcref{Action Map Iso} is a closed embedding of affine $k$-schemes whose ring of functions has the same finite $k$-dimension, it must be an isomorphism.
\end{proof}

\begin{Remark}\label{We actually recover}
Observe that, if the $H$-action on some $x \in V(\overline{k})$ is free, then $C_x = k$, and so \cref{Quotient map iso} recovers the isomorphism \labelcref{Generic Iso} in the introduction.
\end{Remark}

\begin{Corollary}\label{Cartesian Diagram for Stack to GIT Map at K-Point Revamp}
    The diagram \begin{equation*}
  \xymatrix@R+2em@C+2em{
   \text{Spec}(C_x)/H_x \ar[d]_{\dot{x}} \ar[r]^{\pointwisegeneralstacktoGITquotientmap} & \Spec(K) \ar[d]^{\Phi_H \circ x} \\
  V/H \ar[r]^{\generalstacktoGITquotientmap} &  V\sslash H
  }
\end{equation*} is Cartesian, where $\pointwisegeneralstacktoGITquotientmap$ is the terminal map in $K$-schemes.
\end{Corollary}

\begin{proof}
We have isomorphisms \[\Spec(C_x)/H_x\xleftarrow{\sim} (H\backslash H \times \Spec(C_x))/H_x \xrightarrow{\sim} (H \times \Spec(C_x))/(H \times H_x)\xleftarrow{\sim} (H \times^{H_x} \Spec(C_x))/H\] since colimits commute with colimits\footnote{Since colimits are a particular case of left Kan extensions, this fact follows from the essential uniqueness of adjoints.} and that the terminal map induces an equivalence $H/H \xrightarrow{\sim} \ast$. Continuing this chain of isomorphisms, we obtain \[(H \times^{H_x} \Spec(C_x))/H \xrightarrow{\sim} (V \times_{V\sslash H} \Spec(K))/H \xrightarrow{\sim} V/H \times_{V\sslash H} \Spec(K))/H\] by \cref{Quotient map iso} and by using the fact that colimits are universal in an $\infty$-topos such as PreStk \cite[Theorem 6.1.0.6]{LuHTT} and that sheafification commutes with finite limits \cite[Lemma 2.3.6]{GaRoI}, respectively. 
\end{proof}

\begin{proof}[Proof of \cref{Fully Faithfulness and Easy Essential Image of Pullback by generalstacktoGITquotientmap Revamp}] The fact that the diagram of \cref{Cartesian Diagram for Stack to GIT Map at K-Point Revamp} commutes implies that any object $\F$ which descends to $\vectorspaceHactson\sslash H$ satisfies (2). In particular, the fully faithful functor $\generalstacktoGITquotientmap^*$ maps into the full subcategory of objects satisfying (2). Since this functor admits a right adjoint $\generalstacktoGITquotientmap_{\ast}$ (explicitly given by taking the $H$-invariants of the associated complex of $\text{Sym}(\Vdual)$-modules), it remains to verify that the adjoint $\generalstacktoGITquotientmap_{\ast}$ is conservative on the full subcategory of objects satisfying (2). Let $\F$ be such a nonzero object. Then, by \cref{For any qcoh sheaf on smooth classical Noetherian scheme there exists a field valued point where the fiber doesn't vanish}, there exists some field-valued point for which the fiber of $\F$ at this point does not vanish.

We wish to show that $\generalstacktoGITquotientmap_{\ast}(\F)$ does not vanish, and, of course, it suffices to show that $(\generalVectorSpacetoGITQuotientMap \circ x)^*\generalstacktoGITquotientmap_{\ast}(\F)$ does not vanish. Since $x$ is a map of quasicompact schemes, we may apply base change along the diagram of \cref{Cartesian Diagram for Stack to GIT Map at K-Point Revamp} (we may do this, for example, by applying \cite[Proposition 3.10]{BZFNIntegralTransforms}, whose hypotheses hold by \cite[Corollary 3.22, Corollary 3.23]{BZFNIntegralTransforms}) and may equivalently show that $\pointwisegeneralstacktoGITquotientmap_{\ast}\dot{x}^*(\F)$ does not vanish. However, by assumption, $\dot{x}^*(\F)$ lies in the essential image of $\pointwisegeneralstacktoGITquotientmap^*$. Therefore the adjoint $\pointwisegeneralstacktoGITquotientmap_{\ast}\dot{x}^*(\F)$ does not vanish, since the adjoint to a fully faithful functor is conservative on the essential image and $\dot{x}^*(\F)$ is nonzero since $x^*(\F)$ is nonzero. Combining these results, we see that $\generalstacktoGITquotientmap_{\ast}(\F)$ does not vanish on objects satisfying (2), and so the categories given by (1) and (2) are equivalent.

\newcommand{\generalVectorSpacetoStackQuotientMap}{\overline{\phi}}
We now prove that $(2) \Leftrightarrow (3)$. To this end, fix some field-valued point $x: \Spec(K) \to V$, and let $\generalVectorSpacetoStackQuotientMap: V \to V/H$ denote the quotient map. Observe that the following diagram commutes: 

\begin{equation}
  \xymatrix@R+2em@C+2em{
    \text{Spec}(K)/H_x \ar[dr]_{\overline{x}} \ar[r]^{\overline{i}} & \text{Spec}(C_x)/H_x \ar[d]_{\dot{x}} \\
    &  V/ H
  }
\end{equation}
\noindent where $\overline{x}$ is the map induced by the composite $\generalVectorSpacetoStackQuotientMap \circ x$ and $\overline{i}$ is the map induced by the composite of the unique $K$-point $\Spec(K) \to \Spec(C_x)$ of $\Spec(C_x)$ and the quotient $\Spec(C_x) \to \Spec(C_x)/H_x$. In particular, if $\F$ descends to the coarse quotient $\text{Spec}(C_x)\sslash H_x = \text{Spec}(K)$ for some fixed $x$, then we see that $\overline{i}^*(\F) \simeq (\text{Spec}(K)/H_x \to \text{Spec}(K))^*(\mathcal{G})$ for some $\mathcal{G} \in \QCoh(\text{Spec}(K)) = \text{Vect}_K$. Since, under the equivalence $\QCoh(\text{Spec}(K)/H_x) \simeq \text{Rep}_K(H_x)$, this functor corresponds to the inclusion of the trivial $H_x$-representation, we see that if $\F$ satisfies (2) it also satisfies (3). Conversely, if $\F$ satisfies (3) for some fixed $x$, then we see by \cref{Equivariant Sheaf on Coinvariant Algebra Descends to the Coarse Quotient iff Fiber Is Trivial} that $\dot{x}^{\ast}(\F)$ descends to the coarse quotient $\text{Spec}(C_x)\sslash H_x = \text{Spec}(K)$. 
\end{proof}
\begin{Remark}\label{Cant Just Check Closed Points} We claim that, when working with the \textit{cocomplete} category $\QCoh(V)$, condition (3) of \cref{Fully Faithfulness and Easy Essential Image of Pullback by generalstacktoGITquotientmap Revamp} is strictly stronger than the following variant of condition (3):

\begin{enumerate}
    \item[(3')] For each $\overline{k}$-point $x$ of $V$, the $H_x$-representation on $x^{\ast}(\text{oblv}^{H}_{H_x}(\F))$ is trivial.
\end{enumerate}

We now sketch a proof of this fact. Let $H := S_3$ act on $V := \A^3$ by the permutation action; it is not difficult to check that $H$ acts as a reflection group. For each transposition $\sigma \in H$, let $V^{\sigma}$ denote the hyperplane fixed by $\sigma$ and let $i^{\sigma}: V^{\sigma} \to V$ denote the closed embedding of the corresponding $V^{\sigma}$. Define \[M := i_*^{(1, 2)}(K(V^{(1, 2)})) \oplus i_*^{(2, 3)}(K(V^{(2, 3)})) \oplus i_*^{(1, 3)}(K(V^{(1, 3)}))\] where $K(V^{\sigma})$ denotes the field of fractions on the ring of functions on $V^{\sigma}$. It is not difficult to check that the automorphisms \[(1, 2) \cdot (\overline{p}, \overline{q}, \overline{r}) :=  (\overline{(1, 2) \cdot p}, \overline{(1, 2) \cdot r}, \overline{(1, 2) \cdot q})\text{ and } (2, 3) \cdot (\overline{p}, \overline{q}, \overline{r}) :=  (\overline{(2, 3) \cdot r}, \overline{(2, 3) \cdot q}, \overline{(2, 3) \cdot p})\] satisfy the braid relation and thus give rise to an $H$-equivariant structure on $M$. 

If we let $M' := M \otimes_k k_{\mathrm{sign}} \in \QCoh(V)^{H, \heartsuit}$ denote the $H$-equivariant sheaf obtained by twisting the equivariance by the sign character,\footnote{Explicitly, if $\chi: H \to k^{\times}$ denotes the sign character and $e_h: M \xrightarrow{\sim} \mathrm{act}_h^*(M)$ denotes the equivariant structure, then the equivariant structure $M'$ is given by the composite of $e_h$ and multiplication by $\chi(h)$.} the sheaf $M'$ does \textit{not} descend to the coarse quotient, by \cref{Fully Faithfulness and Easy Essential Image of Pullback by generalstacktoGITquotientmap Revamp}. On the other hand, it is elementary to check that $i^*(M') \simeq 0$ for any closed point $i: \Spec(k) \to V$, and so $M'$ certainly satisfies (3'). 
\end{Remark}

\subsection{From QCoh to IndCoh}
We recall the category $\text{IndCoh}(X)$, which is defined if $X$ is any laft prestack, see \cite[Chapter 3, Section 5]{GaRoII}. If $X$ is any laft prestack, there is a canonical symmetric monoidal functor $\Upsilon_{X}: \QCoh(X) \to \IndCoh(X)$ whose underlying functor of DG categories can be identified with $\lq - \otimes_{\mathcal{O}_X} \omega_X$.\rq{} If $X$ is a smooth classical scheme of dimension $d$, $\omega_X$ is the complex which has the canonical sheaf of $X$ in cohomological degree $-d$ and is zero elsewhere, and $\Upsilon_X$ is an equivalence. We show the analogous claim for the quotient stack $X/F$ where $F$ is any finite group. 

\begin{Proposition}\label{Upsilon is an Equivalence for Quotient of Finite Group Acting on Vector Space}
If $F$ is a finite group acting on a smooth scheme $X$ of dimension $d$, then $\Upsilon_{X/F}: \QCoh(X/F) \to \IndCoh(X/F)$ is an equivalence, and $\Upsilon_{X/F}[-d]$ is $t$-exact.
\end{Proposition}

\begin{proof}
The map $q: X \to X/F$ is finite flat. In particular, via flat descent for $\QCoh(X/F)$ \cite[Chapter 3, Section 1.3]{GaRoI}, we obtain that the pullback map induces a canonical equivalence
\raggedbottom
\[q^{\ast}: \QCoh(X/F) \xrightarrow{\sim} \text{lim}_{\Delta}\QCoh(F^{\bullet} \times X)\]

\noindent and since $q$ is in particular proper, proper descent \cite[Chapter 3, Proposition 3.3.3]{GaRoII} gives an analogous equivalence for $\IndCoh(X/F)$. Since, for each integer $i$, the categories $F^{i} \times X$ is again a smooth scheme, we see that $\Upsilon_{F^{i} \times X}$ is an equivalence, and $\Upsilon_{F^{i} \times X}[-\text{dim}(F^{i} \times X)] = \Upsilon_{F^{i} \times X}[-d]$ is a $t$-exact equivalence. Therefore, since we can identify $\Upsilon_{X/F}$ as the composite, read left to right, of the functors
\raggedbottom
\[\QCoh(X/F) \xrightarrow{\sim} \text{lim}_{\Delta}\QCoh(F^{\bullet} \times X) \xrightarrow{\Upsilon_{\bullet}} \text{lim}_{\Delta}\IndCoh(F^{\bullet} \times X) \xleftarrow{\sim} \IndCoh(X/F)\]

\noindent we see that $\Upsilon_{X/F}$ is an equivalence, as desired. 
\end{proof}

We will also use the following analogue of \cref{Equivariant Sheaf on Coinvariant Algebra Descends to the Coarse Quotient iff Fiber Is Trivial} for IndCoh in arguing that sheaves on $\LTd\sslash \Waff$ can be identified with $\Waff$-equivariant sheaves satisfying Coxeter descent, see \cref{Subsubsection with All Equivalent Conditions of Satisfying Descent to Wext Coarse Quotient}:

\begin{Corollary}\label{Pointwise Essential Image of Pullback to Equivariant Sheaves on Coinvariant Algebra for IndCoh} Assume $H$ is a pseudo-reflection group acting on some vector space $V$, and let $C$ denote the coinvariant algebra. \begin{enumerate}
    \item The functors $\Upsilon_{\Spec(C)}$ and $\Upsilon_{\Spec(C)/H}$ are fully faithful.
    \item A given $\F \in \IndCoh(\Spec(C))^H$ descends to the coarse quotient $\Spec(C)\sslash H = \Spec(k)$ if and only if $\F \simeq \Upsilon_{\Spec(C)}(\F')$ for some $\F' \in \QCoh(\Spec(C))$ and the canonical $H$-representation on $i^!(\F)$ is trivial.
\end{enumerate}
\end{Corollary}

\begin{proof}
The fully faithfulness claims follow directly from \cite[Lemma 10.3.4]{GaiIndCoh}.\footnote{We thank an anonymous referee for providing us this reference.}
Because $\Upsilon$ is compatible with the pullback of $\Spec(C)/H \to \Spec(k)$ and $\Upsilon_{\Spec(k)}$ is an equivalence, we see that the condition $\F \simeq \Upsilon_{\Spec(C)}(\F')$ is necessary for a given $\F$ to lie in the essential image. If such an $\F'$ exists, the equivalence stated in (2) follows directly from \cref{Equivariant Sheaf on Coinvariant Algebra Descends to the Coarse Quotient iff Fiber Is Trivial} because $\Upsilon$ is compatible with the pullback of $\Spec(k)/H \to \Spec(C)/H$.
\end{proof}

\subsection{Equivalent Conditions for Descent to Coarse Quotient For Coxeter Groups}\label{Equivalent Conditions for Descent to Coarse Quotient For Coxeter Groups Subsection}
We now give alternate descriptions of those $H$-equivariant sheaves on $V$ which descend to the coarse quotient when $H$ is a finite Coxeter group, \textit{which we assume for this section}. We use IndCoh rather than QCoh, as IndCoh will be the sheaf theory used in the later sections. We note that, by smoothness of $V\sslash H$ (see \cref{Chevalley-Shephard-Todd Theorem}) and by \cref{Upsilon is an Equivalence for Quotient of Finite Group Acting on Vector Space}, all of the analogous results in this section hold when IndCoh is replaced with QCoh. 

We first give an alternate description in the case where $H$ is generated by a single reflection. For a fixed reflection $r \in H$, let $\generalstacktoGITquotientmap_r: V/\langle r \rangle \to V\sslash\langle r \rangle$ denote the quotient map for the action of the order two Weyl group $\langle r \rangle$ acting on $V$, and let $i_r: Z_r := V^{\langle r \rangle} \xhookrightarrow{} V$ denote the inclusion of the closed subscheme of fixed points. Note that $\IndCoh(Z_r/\langle r \rangle) \simeq \IndCoh(Z_r) \otimes \text{Rep}(\langle r \rangle)$ since the action of $r$ is trivial on the fixed point locus.

\begin{Proposition}\label{Trivial Restriction Characterization of Essential Image of Pullback by GeneralStacktoGITMap for Order Two Group}
Let $r$ denote a reflection. An object $\F \in \IndCoh(V/\langle r \rangle)$ descends to the coarse quotient $V\sslash \langle r \rangle$ if and only if the pullback $i_r^!(\F) \in \IndCoh(Z_r) \otimes \text{Rep}(\langle r \rangle)$ lies entirely in the summand indexed by the trivial $\langle r \rangle$-representation.
\end{Proposition}

\begin{proof}
The closed subscheme $Z_r \xhookrightarrow{i_r} V$ and complementary open subscheme $U_r := V\setminus Z_r \xhookrightarrow{j_r} V$ are both affine and induce two Cartesian squares as follows
\raggedbottom
\begin{equation*}
  \xymatrix@R+2em@C+2em{
    Z_r/\langle r \rangle \ar[d]_{\generalstacktoGITquotientmap|_{Z_r}} \ar[r]^{i_r} & \ar[d]_{\generalstacktoGITquotientmap} V/\langle r \rangle & U_r/\langle r \rangle \ar[l]_{j_s} \ar[d]_{\generalstacktoGITquotientmap|_{U_r}}\\
    Z_r\sslash\langle r \rangle  \ar[r]^{i_r} &  V\sslash\langle r \rangle & U_r\sslash\langle r \rangle \ar[l]_{j_s}
  }
\end{equation*}
\noindent where each vertical arrow is obtained from the map $\generalstacktoGITquotientmap$. Since $\generalstacktoGITquotientmap^!$ is fully faithful (\cref{Pullback From GIT to Stack Quotient of Finite Group of Affine Scheme is Fully Faithful}), its essential image is closed under extensions. In particular, an object $\F \in \IndCoh(V/H)$ lies in the essential image if and only if $i_r^!(\F)$ and $j_r^!(\F)$ lie in the essential image of the respective pullbacks. However, since the action of $\langle r \rangle$ is free, the rightmost vertical arrow is an equivalence, so $j_r^!(\F)$ is always in the essential image of $\generalstacktoGITquotientmap^!$. Therefore, $\F$ lies in the essential image if and only if $i_r^!(\F)$ does. 

Note that the action of $r$ on $Z_r$ is trivial, and therefore we see that $Z_r\sslash\langle r \rangle \cong Z_r$. Furthermore, we may identify the pullback $\generalstacktoGITquotientmap|_{Z_r}^!$ with the functor 
\raggedbottom
\[\IndCoh(Z_r) \xrightarrow{\text{id} \otimes \text{triv}} \IndCoh(Z_r) \otimes \text{Rep}(\langle r \rangle)\]

\noindent and so we see that an object of the form $i_r^!(\F)$ is in the essential image of the pullback $\generalstacktoGITquotientmap^!$ if and only if the restriction lies entirely in the trivial summand. Combining this with the assertion that $\F$ lies in the essential image if and only if $i_r^!(\F)$ does, we obtain our desired characterization of the essential image. 
\end{proof}

We now summarize and give various equivalent conditions for a given $\F \in \IndCoh(V)^H$ to descend to the coarse quotient $V\sslash H$, where again we remind that $H$ is a Coxeter group in this section: 

\begin{Proposition}\label{Equivalent Conditions to Be in Essential Image of Pullback of GeneralStackToGITQuotient Pullback}
An object $\F \in \IndCoh(V/H)$ descends to the coarse quotient $V\sslash H$ if and only if one of the following equivalent conditions hold:
\begin{enumerate}
    \item For each field-valued point $x: \Spec(K) \to V$, the pullback $x^!(\text{oblv}^H_{H_x}(\F))$, which canonically acquires a $H_x$-representation in $\text{Vect}_{\fieldpossiblydifferentfromgroundfield}$, is the trivial $H_x$-representation.  
    \item For each reflection $r \in H$, $\text{oblv}_{\langle r \rangle}^H(\F)$ descends to the coarse quotient $V\sslash \langle r \rangle$.
    \item For each simple reflection $s \in H$, $\text{oblv}_{\langle s \rangle}^H(\F)$ descends to the coarse quotient $V\sslash \langle s \rangle$.
    \item Each cohomology group $H^i(\F) \in \IndCoh(V)^{H, \heartsuit} \simeq \IndCoh(V)^{\heartsuit, H}$ lies in the essential image of $\generalstacktoGITquotientmap^!|_{\IndCoh(V\sslash H)^{\heartsuit}}$. 
    \item For every reflection $r \in H$, the cohomology group $\text{oblv}^H_{\langle r \rangle}(H^i(\F)) \in \IndCoh(V)^{\langle r \rangle, \heartsuit}$ lies in the essential image of the pullback $\generalstacktoGITquotientmap_r^!$ restricted to $\IndCoh(V\sslash\langle r \rangle)^{\heartsuit}$ for all $i \in \mathbb{Z}$.
    \item For every simple reflection $s \in H$, the cohomology group $\text{oblv}^H_{\langle s \rangle}(H^i(\F)) \in \IndCoh(V)^{\langle s \rangle, \heartsuit}$ lies in the essential image of the pullback $\generalstacktoGITquotientmap_s^!$ restricted to $\IndCoh(V\sslash\langle s \rangle)^{\heartsuit}$ for all $i \in \mathbb{Z}$.
    \item For each simple reflection $s \in H$, the sheaf $i_s^!(\text{oblv}^W_{\langle s \rangle}(\F)) \in \IndCoh(V^{\langle s \rangle}) \otimes \text{Rep}(\langle s \rangle)$ lies entirely in the summand indexed by the trivial representation. 
    \item For every reflection $r \in H$, the sheaf $i_r^!(\text{oblv}^W_{\langle r \rangle}(\F)) \in \IndCoh(Z_r) \otimes \text{Rep}(\langle r \rangle)$ lies entirely in the summand indexed by the trivial representation. 
\end{enumerate}
\end{Proposition}

\begin{proof}[Proof of \cref{Equivalent Conditions to Be in Essential Image of Pullback of GeneralStackToGITQuotient Pullback}]
Because $\Upsilon$ is an equivalence for $\ast/H_x$ and $V/H$ and intertwines $\ast$-pullback for $\QCoh$ and $!$-pullback for $\IndCoh$, we see that that $\F$ descends to the coarse quotient if and only if $\F$ satisfies (1) by \cref{Fully Faithfulness and Easy Essential Image of Pullback by generalstacktoGITquotientmap Revamp}. Similarly, we see that $\F$ descends to the coarse quotient if and only if $\F$ satisfies (2) \cref{Equivariant Sheaf for Pseudo-Reflection Group Descends iff It Descends for Each Reflection}.

We now show $(3) \Rightarrow (2)$. Fix some reflection $r \in H$, and choose some $w \in H$ for which $w^{-1}rw$ is a simple reflection $s$. Then the following diagram commutes
\raggedbottom
\begin{equation*}
  \xymatrix@R+2em@C+2em{
   V\sslash\langle r \rangle \ar[d]_{w} & \ar[l]_{\generalstacktoGITquotientmap_r}\ar[d]_{w} V/\langle r \rangle \ar[dr] \\
  V\sslash\langle s \rangle & \ar[l]_{\generalstacktoGITquotientmap_s} V/\langle s  \rangle \ar[r] & V/W
  }
\end{equation*}

\noindent where the vertical arrows are the maps induced by the action of $w \in H$ and the unlabeled arrows are the quotient maps. We then see if $\text{oblv}^W_{\langle
s \rangle}(\F) \simeq \generalstacktoGITquotientmap_s^!(\F')$ for some $\F'$ then 
\raggedbottom
\[\text{oblv}^W_{\langle r \rangle}(\F) \simeq w^!\text{oblv}_{\langle s \rangle}^{W}(\F) \simeq w^!(\generalstacktoGITquotientmap_s^!(\F')) \simeq \generalstacktoGITquotientmap_r^!(w^!(\F'))\]

\noindent showing that (2) holds. Conversely, the implication $(2) \Rightarrow (3)$ follows since simple reflections are reflections.

The equivalence of a given $\F$ descending to the coarse quotient $V\sslash H$ and the given $\F$ satisfying $(4)$ follows from the $t$-exactness and fully faithfulness of $\generalstacktoGITquotientmap^!$, where the $t$-exactness follows from the fact that $\text{oblv}^W$ reflects the $t$-structure and the fact that the quotient map $(V \to V\sslash H)$ is finite-flat, see \cref{For Pseudo-Reflection Groups Acting Faithfully on Vector Space The GIT Quotient is Flat}. Replacing the map $\generalstacktoGITquotientmap^!$ with $\generalstacktoGITquotientmap_r^!$ and $\generalstacktoGITquotientmap_s^!$, this argument also gives the equivalences $(2) \Leftrightarrow (5)$ and $(3) \Leftrightarrow (6)$. Finally, the equivalences $(2) \Leftrightarrow (7)$ and $(3) \Leftrightarrow (8)$ follow directly from \cref{Trivial Restriction Characterization of Essential Image of Pullback by GeneralStacktoGITMap for Order Two Group}.
\end{proof}

\section{Preliminary Computations}\label{The Union of Graphs of Affine Weyl Group is Ind Finite-Flat Section}
In this section, we review some preliminary results on the union of graphs of the affine Weyl group $\Waff$ and discuss extensions to the extended affine Weyl group. 


\subsection{The Integral Weyl Group}\label{IntegralWeylGroupDefinitionSection}
Let $\varphi: \widetilde{G}_{\mathrm{der}} \to G_{\mathrm{der}}$ denote the simply connected covering of the derived subgroup $G_{\mathrm{der}}$ of $G$. We recall that the natural map $\tilde{m}$ defined as the composite \[\tilde{G} := \widetilde{G}_{\mathrm{der}} \times Z(G)^{\circ} \xrightarrow{(\varphi \times \mathrm{id})} G_{\mathrm{der}} \times Z(G)^{\circ} \xrightarrow{m} G\] is a multiplicative isogeny, where $m$ is the multiplication map. Indeed, this follows, for example, from the fact that the quotient map of $\widetilde{G}_{\mathrm{der}}$ onto the derived subgroup of $G$ is a multiplicative isogeny by construction (see for example \cite[Definition 18.7]{MilneAlgebraicGroupsBook}) and the fact that the multiplication map $m: G_{\mathrm{der}} \times Z(G)^{\circ} \to G$ is a multiplicative isogeny by say \cite[Proposition 14.2]{BorelLinearAlgebraicGroups}. Let $T_0 := T \cap G_{\mathrm{der}}$. We recall that (by say \cite[Theorem 22.6]{BorelLinearAlgebraicGroups}) the group \[\tilde{T} := \tilde{m}^{-1}(T) \cong \varphi^{-1}(T_0) \times Z(G)^{\circ}\] is also a torus and we may identify the Weyl group of $\widetilde{G}_{\mathrm{der}} \times Z(G)$ with $W$ in such a way that the map $\tilde{m}|_{\tilde{T}}: \tilde{T} \to T$ is $W$-equivariant and pullback, respectively pushforward, by $\tilde{m}|_{\tilde{T}}$ gives a bijection from the roots for the $T$-action on $\LG$, respectively coroots for the $\tilde{T}$-action on $\mathrm{Lie}(\tilde{G})$, to the set of roots for the $\tilde{T}$-action on $\mathrm{Lie}(\tilde{G})$, respectively the coroots for the $T$-action on $\LG$. Moreover, since $\tilde{m}$ is an isogeny, its kernel is finite, and so the induced map $\characterlatticeforT \to X^{\bullet}(\tilde{T})$ on the characters of these tori is injective and its cokernel is finite. In particular, there is an isomorphism \begin{equation}\label{The rationalization map}\characterlatticeforT_{\Q} := \Q \otimes_{\mathbb{Z}} \characterlatticeforT \xrightarrow{\sim} \Q \otimes_{\mathbb{Z}} X^{\bullet}(\tilde{T}) =: X^{\bullet}(\tilde{T})_{\Q}\end{equation} of rational vector spaces induced by pullback. Let $\Lambda := X^{\bullet}(\tilde{T})$. Observe that, under the identification of \labelcref{The rationalization map}, we have \begin{equation}\label{Lambda Contained in Natural Pairing with Coroots is Integers}\Lambda \subseteq \{\lambda \in \characterlatticeforT_{\Q} : \lambda(\alpha^{\vee}) \in \mathbb{Z}\text{ for all coroots }\alpha^{\vee}\}\end{equation} since pushforward by $\tilde{m}|_{\tilde{T}}$ induces a bijection on coroots. 

\begin{Remark} To orient the reader on the notation, we record two facts without proof, neither of which will be used in what follows:
\begin{enumerate}
    \item If $G$ is semisimple one can show that the containment \labelcref{Lambda Contained in Natural Pairing with Coroots is Integers} is an equality.
    \item If $G$ is semisimple, one can show that $\characterlatticeforT = \Lambda$ if and only if $G$ is simply connected.
\end{enumerate}
\end{Remark}

For a fixed $x \in \LTd(K)$, let $[x]$ denote the image of this $K$-point in the quotient $\LTd/\Lambda$. The following is essentially shown in the proof of \cite[Satz 1.3]{JantzenModulnMitEinehHochstenGewicht}; we recall some details for the convenience of the reader: 

\begin{Proposition}\label{Equivalent Notions of Integral Weyl Group}
Fix some $x \in \LTd(K)$. The following subgroups of $W$ are identical:
\begin{enumerate}
    \item The image $\overline{\Waff_x}$ of the stabilizer $\Waff_x$ of $x$ under the $\Waff$-action on $\LTd$ under the quotient map $\Waff \to \Waff/\rootlattice \cong W$. 
    \item The subgroup $W_{[x]} := \{w \in W : wx - x \in \rootlattice\}$.
    \item The subgroup $W^{\bullet}_{[x]} := \{w \in W : w\cdot x - x \in \rootlattice\}$. 
     \item The subgroup $W^{[x]} := \langle s_{\alpha} : \langle x , \alpha^{\vee} \rangle \in \mathbb{Z}\rangle$, where $\alpha$ varies over the set of roots $\Phi$. 
\end{enumerate}

\noindent Furthermore, the group $W^{[x]}$ is a Weyl group of the root system whose roots are $\Phi_{[x]} := \{\alpha \in \Phi: \langle x, \alpha^{\vee} \rangle \in \mathbb{Z}\}$. 
\end{Proposition}

\begin{proof}
Observe that, since pullback by $\tilde{m}|_{\tilde{T}}$ gives an isomorphism on the root lattice, there is a canonical isomorphism between the affine Weyl groups for $G$ and $\tilde{G}$ such that the isomorphism $\mathrm{Lie}(\tilde{T})^* \xleftarrow{\sim} \LTd$ induced by $\tilde{m}|_{\tilde{T}}$ is $\Waff$-equivariant. Therefore the groups of (1)-(4) for the group $G$ agree with their natural analogues where $G$ is replaced with $\tilde{G}$. Therefore, we may assume that $G$ is the product of a simply connected semisimple group and a torus.

We first show that $\overline{\Waff_x} = W_{[x]}$. Choose a $\mathbb{Q}$-basis for $K$, which induces a direct sum decomposition for the $\mathbb{Q}$-vector space $\LTd(K)$. In particular, we may write $x = \sum_{i = 0}^dq_ix_i$, where $x_0$ lies in the $\mathbb{Q}$-span of the roots and the $x_i$ for $i > 0$ do not and $q_i \in \mathbb{Q}$. Note that $W$ preserves this direct sum decomposition. 


Assume $w \in \overline{\Waff_x}$. Then there is some $\mu \in \rootlattice$ such that the affine Weyl group element $\tau_{\mu}w$ fixes $x$, where $\tau_{\mu} \in \Waff$ denotes the element which translates by $\mu$. Using the $W$-invariant decomposition above, we see that if $i > 0$, $wx_i = x_i$ and that $\tau_{\mu}wx_0 = x_0$. In particular, $wx - x = wx_0 - x_0 = -\mu$ lies in the root lattice, so we see $w \in W_{[x]}$. Conversely, assume that $w \in W_{[x]}$. Then because $wx - x \in \rootlattice$, the direct sum decomposition above implies that $wx_i = x_i$ for all $i > 0$ and $wx_0 - x_0 = \nu$ for some $\nu \in \rootlattice$. This in particular implies that $(\tau_{-\nu}, w)x = x$, so that $(\tau_{-\nu}, w) \in \Waff_x$ and thus $w \in \overline{\Waff_x}$. 

To show $W^{\bullet}_{[x]} = W_{[x]}$, we first note that for $w \in W$ and $x \in \LTd(K)$, $w\cdot x - x = wx - x + w\rho - \rho$. Therefore, our desired equality follows from the fact that $w\rho - \rho$ lies in the root lattice. In fact, this is more generally true for any $\nu \in \Lambda$ and any $w \in W$. This is because if $w$ is a simple reflection associated to a root $\alpha$ because $w\nu - \nu = -\langle \alpha^{\vee}, \nu \rangle\alpha$, which lies in the root lattice by containment \labelcref{Lambda Contained in Natural Pairing with Coroots is Integers}, and follows for general $w \in W$ by writing $w = s_1...s_r$ and noting that
\raggedbottom
$$w\nu - \nu = (s_1(s_2...s_r\nu) - s_2...s_r\nu) + (s_2(s_3...s_r\nu) - s_3...s_r\nu) + ... + (s_1\nu - \nu)$$

\noindent and that the Weyl group preserves $\Lambda$.  

Finally, the equality $W^{[x]} = W_{[x]}$ and the final statement are precisely the content of \cite[Satz 1.3]{JantzenModulnMitEinehHochstenGewicht} when $G$ is semisimple: its proofs naturally extend to the case where $G$ is a product of a semisimple group and a torus. 
\end{proof}

\begin{Definition}\label{Integral Weyl Group Definition}
For a fixed $x \in \LTd(K)$, we will refer to the group $W_{[x]}$ as the \textit{integral Weyl group} associated to $x$. 
\end{Definition}

\begin{Remark}\label{Integral Weyl Group is Subgroup of Affine Weyl Group Remark}
Note that we may also realize the integral Weyl group associated to $x$ as a subgroup of $\Waff$. In other words, in the notation of \cref{Equivalent Notions of Integral Weyl Group}, we have that the quotient map induces an isomorphism $\Waff_x \cong \overline{\Waff_x}$. This follows because if $w \in \overline{\Waff_x}$, there exists a unique $\mu \in \rootlattice$ such that $(\tau_{\mu}, w)x = x$; the existence is given in the second paragraph of the proof of \cref{Equivalent Notions of Integral Weyl Group} by taking $\mu := -\nu$, and the uniqueness follows trivially since $wx + \mu_1 = wx + \mu_2$ if and only if $\mu_1 = \mu_2$. 
\end{Remark}

\subsection{The Union of Graphs of a Finite Closed Subset of $\Waff$ is Finite Flat}
We now prove the following possibly known result which we were unable to locate a reference for:

\begin{Proposition}\label{FiniteFlat}
Let $S \subseteq W^{\text{aff}}$ denote a finite, closed (see \cref{ClosedSubsetDefinition}) subset of the \text{affine} Weyl group $\Waff$, and let $\pi_S: \Gamma_S \xhookrightarrow{} \LTd \times \LTd$ denote the union of graphs of those $w \in S$. Then the projection map onto the first factor $s: \Gamma_S \to \LTd$ is finite flat. 
\end{Proposition}
As we will see, \cref{FiniteFlat} will follow from a known extension of Borel's theorem on the cohomology of the flag variety, which we include an alternate proof of in the appendix. Using the results of the appendix, we prove \cref{FiniteFlat} after proving the following lemma: 

\begin{Lemma}\label{Picking a Field-Valued Point Splits Off Different Waff Orbits}
Fix some finite subset $S \subseteq \Waff$, and fix some $\lambda \in \LTd(\fieldpossiblydifferentfromgroundfield)$ for $\fieldpossiblydifferentfromgroundfield$ a (classical) field. Then, if $\Waff_{\lambda} \leq \Waff$ denotes the stabilizer of $\lambda$, the coproduct of inclusions induces a canonical isomorphism:

$$\Gamma_S \times_{\LTd} \Spec(\fieldpossiblydifferentfromgroundfield) \simeq \coprod_{w\Waff_{\lambda}  \in \Waff/\Waff_{\lambda} } (\Gamma_{S \cap w\Waff_{\lambda} } \times_{\LTd} \Spec(\fieldpossiblydifferentfromgroundfield))$$
\end{Lemma}

\begin{proof}
Enumerate the $\Waff$-orbit of $\lambda$ as $\{\lambda_i\}_{i \in \mathbb{N}}$, we can partition the set $S = \cup_{i}S_i$ where all $w_i \in S_i$ have the property that $w_i\lambda = \lambda_i$. Let $Z_S$ denote the closed subscheme of $\LTd$ given by $\bigcup_{i, j \in \mathbb{N}, i \neq j}\bigcup_{u \in S_i, v \in S_j}\{w_i = w_j\}$; since $S$ is finite, $Z_S$ can be expressed as a finite union of nonempty Zariski closed subsets. Furthermore, we see that $\lambda$ factors through the open complement $U_s$ of $Z_S$. Therefore, since $\Gamma_S \times_{\LTd} \Spec(\fieldpossiblydifferentfromgroundfield) \simeq (\Gamma_S \times_{\LTd} U_S) \times_{U_S} \text{Spec}(K)$ and, by definition of $U_S$, we have that $(\Gamma_S \times_{\LTd} U_S)$ can be written as a disjoint union of subschemes indexed by each $S_i$, we see that our induced map is an isomorphism. 
\end{proof}

\begin{proof}[Proof of \cref{FiniteFlat}]
The map $(\text{act}, \text{proj}): S \times \LTd \to \Gamma_S$ is dominant. Therefore $\mathcal{O}(\Gamma_S)$ is finitely generated as a $\Symt$-module, as it is a submodule of the finitely generated $\Symt$-module $\mathcal{O}(S \times \LTd)$ and $\Symt$ is Noetherian. It therefore remains to show that $s$ is flat. Since $s$ is a surjective morphism onto a smooth variety and every irreducible component of $\Gamma_S$ has the same dimension as $\LTd$, it is enough to show that for all points $x$ in $\LTd$, the length of the fiber at $x$ as an $\mathcal{O}_{\Gamma_{S}}$ module is independent of $x$ by a standard result in commutative algebra, see for example \cite[Corollary 18.17]{Eis}. Choose some field-valued $x \in \LTd(K)$. We enumerate the distinct $K$-points in the fiber of $x$, say, $(x, y_1), ..., (x, y_m) \in \Gamma_S$, and write $y_i = w_ix$ for the minimal such $w_i \in W^{\text{aff}}$. Write $\Waff = M\Waff_x$ where $M$ denotes the set of minimal elements of each coset in $\Waff/\Waff_x$. We thus have an isomorphism given by \cref{Picking a Field-Valued Point Splits Off Different Waff Orbits}

$$\Gamma_S \times_{\LTd} \{x\} \cong \coprod_{i}(\Gamma_{w_{i}\Waff_x \cap S} \times_{\LTd} \{x\})$$

\noindent where $i$ ranges over a finite index set and $w_i \in M$. Note that each $\Gamma_{w_{i}\Waff_x \cap S} \times_{\LTd} \{x\}$ is isomorphic via left multiplication by $w_i^{-1}$ to the fiber $\Gamma_{S'} \times_{\LTd} \Spec(K)$ for some subset $S' \subseteq \Waff_x$. 

Furthermore, we claim this $S'$ is a closed subset of $\Waff_x$ in the sense of \cref{ClosedSubsetDefinition}. To see this, recall the canonical isomorphism $\Waff_x \cong \overline{\Waff_x} = W^{[x]}$ given in \cref{Integral Weyl Group is Subgroup of Affine Weyl Group Remark} and \cref{Equivalent Notions of Integral Weyl Group}. If $u \in S'$ and $u' \leq^x u$ (where $\leq^x$ refers to the ordering on the Coxeter group $W^{[x]}$), we have that by \cite[Lemma 2.5]{LusztigMonodromicSystemsonAffineFlagManifolds}, $w_iu' \leq w_iu$, so that the fact that $S$ is closed gives that $u' \in S'$, and so $S'$ is closed in $W^{[x]}$. Therefore we can apply \cref{ExplicitBasis} to see that the total length of the fiber is $\sum_{i = 1}^m|S \cap w_iW_x| = |S|$, which is independent of $x$. 
\end{proof}

\subsection{Fiber of Map to Coarse Quotient of Affine Weyl Group at Field-Valued Point}
Fix some field $K/k$ and let $x \in \LTd(K)$. We will use the following lemma later:  

\begin{Lemma}\label{Orbit-Stabilizer for Union of All Graphs}
The multiplication map induces a left $\Waff$-equivariant isomorphism 
\raggedbottom
\[\eta: \Waff \mathop{\times}\limits^{\Waff_{x}} (\Gamma_{\Waff_{x}} \times_{\LTd} \Spec(K)) \xrightarrow{\sim} \Gamma_{\Waff} \times_{\LTd} \Spec(K)\]

\noindent where $\Gamma_{\Waff_{x}}$ is the union of graphs of the subgroup $\Waff_{x} \leq \Waff$, which admits a map $\Gamma_{\Waff_{x}} \to \LTd$ given by the projection $(wx, x) \mapsto x$.  
\end{Lemma}

\begin{proof}
Using the results of \cref{Picking a Field-Valued Point Splits Off Different Waff Orbits} for every finite subset $S \subseteq \Waff$, we obtain an isomorphism 
\raggedbottom
\[\Gamma_{\Waff} \times_{\LTd} \Spec(K) \simeq \bigsqcup (\Gamma_{w\Waff_{x}} \times_{\LTd} \Spec(K))\]

\noindent where the right hand side ranges over the cosets of $\Waff/\Waff_{x}$. Therefore we may check that the multiplication map induces an isomorphism at each open subset $\Gamma_{w\Waff_{x}} \times_{\LTd} \Spec(K)$. However, we see that taking the fiber product of the above multiplication map by this open subset, we obtain the map
\raggedbottom
\[w\Waff_{x} \mathop{\times}\limits^{\Waff_{x}} (\Gamma_{\Waff_{x}} \times_{\LTd} \Spec(K)) \xrightarrow{} \Gamma_{w\Waff_{x}} \times_{\LTd} \Spec(K)\]

\noindent which is an isomorphism. Furthermore, $\eta$ is $\Waff$-equivariant because the multiplication map is $\Waff$-equivariant. 
\end{proof}

\section{The Coarse Quotient for the Affine Weyl Group}\label{The Coarse Quotient for Affine Weyl Group Section}
In this section, we define the coarse quotient $\LTd\sslash\Wext$ and determine some of its basic properties. After briefly reviewing the notion of a groupoid object in \cref{Groupoid Objects and Higher Algebra Subsection}, we define this quotient in general in \cref{Coarse Quotient Definition Section}, and show that sheaves on $\LTd\sslash\Wext$ are equivalently $\Wext$-equivariant sheaves satisfying conditions analogous to those of \cref{Equivalent Conditions to Be in Essential Image of Pullback of GeneralStackToGITQuotient Pullback} in \cref{CoxeterDescentSubsection}. 

In order to compute the category of sheaves on $\LTd\sslash\Wext$, we will use the fact that the map $\LTd \to \LTd\sslash\Wext$ is an ind-finite flat cover. Using the computations of \cref{The Union of Graphs of Affine Weyl Group is Ind Finite-Flat Section}, we show this in \cref{WAffIndFiniteFlat}. 


\subsection{Groupoid Objects and Higher Algebra}\label{Groupoid Objects and Higher Algebra Subsection}
In this section, we briefly recall the notion of groupoid objects in the higher categorical context. For a thorough treatment of groupoid objects in the (1,1) and $(\infty, 1)$ setting, see \cite[Section 6.1.2]{LuHTT}. 




In derived algebraic geometry, the notion of a classical groupoid is replaced with the notion of a \textit{groupoid} or an $\infty$\text{-groupoid} of an $(\infty, 1)$-category. We use an equivalent formulation of the notion of a groupoid (see \cite[Section 6.1.2.6]{LuHA}):

\begin{Definition}\label{Definition of Infinity Groupoid}
A \textit{groupoid object} of an $(\infty, 1)$ category $\C$ is a simplicial object $U$ of $\C$ such that for every $n \geq 0$ and every partition $[n] = S \cup S'$ such that $S \cap S'$ consists of a single element $s$, the canonical map $U([n]) \to U(S) \times_{U(\{s\})} U(S')$ is an equivalence (and, in particular, the latter term is defined). 
\end{Definition}

We now recall the basic results about groupoid objects in the $(\infty, 1)$ category of spaces Spc, which immediately implies the analogous fact for the category of prestacks since the results of \cite[Chapter 5.1.2]{LuHTT} show that limits and colimits in functor categories are computed termwise:

\begin{Proposition}\cite[Corollary 6.1.3.20]{LuHTT}\label{GroupoidResults}
Every groupoid object of Spc is \textit{effective}. In particular, if $U_{\bullet}$ is a groupoid object of Spc, then a geometric realization $U_{-1}$ of it exists and the canonical map $U_1 \to U_0 \times_{U_{-1}} U_0$ is an equivalence. 
\end{Proposition}



\subsection{The Coarse Quotient}\label{Coarse Quotient Definition Section}
We now wish to apply the general framework above to our specific case of interest. Let $\Gamma_{\Waff}$ denote the union of graphs of each $w \in \Waff$, where we recall by graph we mean the closed subschemes $\LTd \xhookrightarrow{(w, \text{id})} \LTd \times \LTd$. More precisely, we view $\Gamma_{\Waff}$ as the classical ind-scheme given by the union of graphs $\Gamma_S$ given by the intersection of ideals (see \cref{UnionNotProductExample}) for $S$ a finite closed subset of the affine Weyl group. In particular $\Gamma_{\Waff}$ is naturally an ind-closed subscheme of $\LTd \times \LTd$. 

\subsubsection{Definition of the Coarse Quotient}\label{Definition of Coarse Quotient Subsubsection} Let $\Gamma_{\Wext}$ denote the balanced product $\Wext \mathop{\times}\limits^{\Waff} \Gamma_{\Waff}$. Note that $\Gamma_{\Wext}$ admits canonical maps to $s,t: \Gamma_{\Wext} \to \LTd$ given by the maps $s(\sigma, (w\lambda, \lambda)) = \lambda$ and $t(\sigma, (w\lambda, \lambda)) = \sigma w\lambda$. We now record the following general fact which remains valid if the subgroup $\Waff \leq \Wext$ is replaced with any closed subgroup $H \leq \tilde{H}$ and $\Gamma_{\Waff}$ is replaced with any $\Gamma$ with an $H$-action:

\begin{Proposition}\label{Action Map is Projection Map}
We have an isomorphism $\Gamma_{\Wext} \xrightarrow{\sim} \Wext/\Waff \times \Gamma_{\Waff}$ in such a way that the following diagram commutes: 

    \begin{equation*}
  \xymatrix@R+2em@C+2em{
\Gamma_{\Wext} := \Wext \mathop{\times}\limits^{\Waff} \Gamma_{\Waff} \ar[dr]_{t} \ar[r]^{\sim}  & \Wext/\Waff \times \Gamma_{\Waff} \ar[d]_{\text{proj}}\\
& \LTd
  }
 \end{equation*}
\end{Proposition}

\begin{proof}
The isomorphism is induced by the map $(\tilde{w}, g) \mapsto (\tilde{w}, \tilde{w}g)$, and the inverse map is induced by $(\tilde{w}, g') \mapsto (\tilde{w}, \tilde{w}^{-1}g')$. 
\end{proof}

Using the formulas of \labelcref{Groupoid Formulas}, we may construct a groupoid object $\Gamma_{\bullet}$ over $\LTd$ such that $\Gamma_1 \simeq \Gamma_{\Wext}$--specifically, set $\Gamma_n := \Gamma_{\Wext} \times_{\LTd} ... \times_{\LTd} \Gamma_{\Wext}$. However, working with this object in general would be a technical nusance, a priori: all of our fiber products are inherently derived. However, the following proposition allows us to argue that the $\Gamma_n$ systematically remain in the classical (1,1)-categorical setting. 

\begin{Proposition}\label{GroupoidIndFiniteFlat}
The map $s: \Gamma_{\Wext} \to \LTd$ is ind-finite flat. 
\end{Proposition}

\begin{proof}
By \cref{Action Map is Projection Map}, it suffices to show that the map is ind-finite flat flat when $\Wext = \Waff$. In this case, $\Waff$ is a Coxeter group and has a length function $\ell$. For each positive integer $m$ set $S_m = \{w \in W^{\text{aff}} : \ell(w) \leq m\}$,  and let $G_m$ denote the union of graphs of those $w \in S_m$. Then we clearly have $\Gamma = \cup_mG_m$, and so it suffices to show that $s_m: G_m \to \LTd$ is finite flat. However, this follows from \cref{FiniteFlat}. 
\end{proof}

\begin{Corollary}\label{All Objects of Groupoid Gammabullet Are Classical Indschemes}
Each $\Gamma_{i}$ is a filtered colimit of classical schemes.
\end{Corollary}





\begin{Definition}\label{Definition of Coarse Quotient}
    We define the \textit{coarse quotient} as the prestack $\LTd\sslash\Wext$ obtained from the geometric realization of $\Gamma_{\bullet}$.
\end{Definition}

Observe that, when $\Wext = \Waff$, \cref{Definition of Coarse Quotient} recovers \cref{Definition of Coarse Quotient for Coxeter Group}. Observe moreover that \cref{GroupoidResults} immediately implies the natural analogue of our desirada \labelcref{Union of Graphs Is Product Over GIT Quotient} from the introduction:

\begin{Proposition}\label{SpecificGroupoidResults}
We have a canonical equivalence $\Gamma_{\Wext} \xrightarrow{\sim} \LTd \times_{\LTd\sslash\Wext} \LTd$. 
\end{Proposition}

This has the following corollary, which should be compared to the finite group case in \cref{Quotient map iso}. 

\begin{Corollary}\label{Soergel-Ish Isomorphism}
For a fixed $x \in \LTd(K)$, we have canonical isomorphisms
\raggedbottom
\[\characterlatticeforT\backslash \LTd \times_{\LTd\sslash\Wext} \Spec(K) \simeq \characterlatticeforT\backslash \Gamma_{\Wext} \times_{\LTd} \Spec(K) \xrightarrow{\sim} W \mathop{\times}\limits^{W_{[x]}} \Spec(C_{W_{[x]}})\]

\noindent where $C_{[x]}$ denotes the coinvariant algebra for the integral Weyl group (\cref{Integral Weyl Group Definition}) of $x$. 
\end{Corollary}

\begin{proof}
The first equivalence follows from direct application of \cref{SpecificGroupoidResults}; we now show the second. Note that the inclusion $\Waff \xhookrightarrow{} \Wext$ induces an isomorphism $\rootlattice\backslash \Waff \xrightarrow{\sim} \characterlatticeforT\backslash \Wext$. Therefore we obtain canonical isomorphisms
\raggedbottom
\[\characterlatticeforT\backslash \Gamma_{\Wext} :=  \characterlatticeforT\backslash \Wext \mathop{\times}\limits^{\Waff} \Gamma_{\Waff} \xleftarrow{\sim} \rootlattice\backslash \Waff \mathop{\times}\limits^{\Waff} \Gamma_{\Waff} \simeq \rootlattice\backslash \Gamma_{\Waff}\]
\noindent over $\LTd$ with respect to the (right) projection map. Furthermore, by \cref{Orbit-Stabilizer for Union of All Graphs}, we see:
\raggedbottom
\[\rootlattice\backslash \Gamma_{\Waff} \times_{\LTd} \Spec(K) \xleftarrow{\sim}  \rootlattice\backslash \Waff \mathop{\times}\limits^{\Waff_{x}} (\Gamma_{\Waff_{x}} \times_{\LTd} \Spec(K))\]

\noindent so that, composing with the quotient map $\rootlattice\backslash \Waff \xrightarrow{\sim} W$, we see the quotient map induces an isomorphism
\raggedbottom
\[\rootlattice\backslash \Waff \mathop{\times}\limits^{\Waff_{x}} (\Gamma_{\Waff_{x}} \times_{\LTd} \Spec(K)) \xrightarrow{\sim} W \mathop{\times}\limits^{W_{[x]}} \Gamma_{\Waff_{x}} \times_{\LTd} \Spec(K)\]

\noindent where we identify the image of the stabilizer $\Waff_{x}$ with the integral Weyl group as in \cref{Integral Weyl Group is Subgroup of Affine Weyl Group Remark}. Therefore, since we may identify 
\raggedbottom
\[W \mathop{\times}\limits^{W_{[x]}} \Gamma_{\Waff_{x}} \times_{\LTd} \Spec(K) \simeq W \mathop{\times}\limits^{W_{[x]}} \Spec(C_{x}) \cong W \mathop{\times}\limits^{W_{[x]}} \Spec(C_{[x]})\]

\noindent where the first isomorphism follows from the definition of $C_{[x]}$ and the second isomorphism is given by translation. 
\end{proof}


We once and for all let $\quotientmapforcoarsequotient: \LTd \to \LTd\sslash\Wext$ denote the quotient map. 

\begin{Corollary}\label{WAffIndFiniteFlat}
The map $\quotientmapforcoarsequotient: \LTd \to \LTd\sslash\Wext$ is ind-finite flat. 
\end{Corollary}

\newcommand{\colim}{\text{colim}}
\begin{proof}
Observe that if we are given an $S$-point $x: S \to \LTd\sslash\Wext$ for some $S \in \mathrm{Sch}^{\mathrm{aff}}$, then $x$ factors through some point of $\LTd(S)$: indeed, this follows since there is an equivalence \begin{equation}\label{Colim Iso}\mathrm{Maps}_{\text{PreStk}}(S, \LTd\sslash\Wext) \xleftarrow{\sim} \colim_{\Delta^{op}} \mathrm{Maps}_{\text{PreStk}}(S, \Gamma_{\bullet})\end{equation} given by the obvious map\footnote{This map will not be an equivalence if we replace $\LTd\sslash\Wext$ with its sheafification with respect to some Grothendieck topology; see \cref{Comparison to Classical GIT Construction Subsubsection} for further discussion on the comparison between $\LTd\sslash\Wext$ and its various sheafifications.} and so we obtain an induced equivalence \[\pi_0\mathrm{Maps}(S, \LTd\sslash\Wext) \xleftarrow{\sim} \colim_{\Delta^{op}} \pi_0(\mathrm{Maps}(S, \Gamma_{\bullet}))\] since $\pi_0$ is a left adjoint and thus commutes with colimits. Since the map \[\pi_0(\mathrm{Maps}(S, \LTd)) = \pi_0(\mathrm{Maps}(S, \Gamma_{0})) \to  \colim_{\Delta^{op}} \pi_0(\mathrm{Maps}(S, \Gamma_{\bullet}))\] of sets is obviously surjective, choosing a point in the preimage of this map gives our desired lift.

Therefore, to show that the map $S \times_{\LTd\sslash \Wext} \LTd \to S$ is ind-finite flat, it suffices to show the map $\LTd \times_{\LTd\sslash\Wext} \LTd \to \LTd$ is finite flat. However, by \cref{GroupoidResults}, our based changed map is canonically $t: \Gamma \to \LTd$, which is ind-finite flat by \cref{GroupoidIndFiniteFlat}. 
\end{proof}

In what follows, our main goal is to study the category $\IndCoh(\LTd\sslash\Wext)$. To show IndCoh is defined on $\LTd\sslash\Wext$, we prove the following: 

\begin{Proposition}\label{Affine Coarse Quotient is 0-Coconnective LFT}
Any prestack which is a countable discrete set of points is $0$-coconnective locally finite type (lft). Furthermore, the respective quotient prestacks $\LTd/\characterlatticeforT$ and $\LTd\sslash\Wext$ are $0$-coconnective lft prestacks.
\end{Proposition}

\begin{proof}
The $n$-coconnective prestacks (respectively, the $n$-coconnective lft prestacks) are those prestacks which are in the essential image of a certain left adjoint--namely, the left Kan extension of the inclusion of $n$-coconnective affine schemes (respectively, the left Kan extension of the inclusion of $n$-coconnective finite type affine schemes in the sense of \cite[Chapter 2, Section 1.5]{GaRoI}). Therefore, the condition of being $n$-coconnective and the condition of being $n$-coconnective lft are conditions closed under colimits. Classical finite type affine schemes are $0$-coconnective lft. Therefore since this condition is closed under colimits, then any colimit of classical schemes is 0-coconnective lft. In turn, since $\LTd\sslash\Wext$ is a certain colimit of 0-coconnective lft prestacks by \cref{All Objects of Groupoid Gammabullet Are Classical Indschemes}, we see that $\LTd\sslash\Wext$ is a 0-coconnective lft prestack. 
\end{proof}

\subsubsection{Comparison to Classical GIT Construction}\label{Comparison to Classical GIT Construction Subsubsection}
In \cref{Comparison to Classical GIT Construction Subsubsection}, we discuss a potential alternate definitions of the coarse quotient: the fppf sheafification of the prestack $\LTd\sslash\Wext$ of \cref{Definition of Coarse Quotient}. The results and discussion in \cref{Comparison to Classical GIT Construction Subsubsection} will not be used elsewhere in this paper.

One advantage of taking the fppf sheafification of $\LTd\sslash\Wext$ as the definition of the coarse quotient is that this more closely parallels the construction of $\LTd\sslash W$ in a way we now make precise. By using essentially identical methods as in the construction of the groupoid object $\Gamma_{\bullet}$ above, one can construct a finite flat groupoid $\Gamma_{\bullet}^{\mathrm{fin}}$ for which $\Gamma_{1}^{\mathrm{fin}}$ is the union of graphs of the Weyl group acting on $\LTd$, and define a prestack $\mathcal{Q}$ by taking the geometric realization of this groupoid object in the category of prestacks. Let $L$ denote the sheafification functor with respect to the fppf topology. Recall that the affine scheme $\LTd\sslash W$ is a sheaf with respect to the fppf topology \cite[Lemma 2.2.2.13]{ToenVezzosiHomotopicalAlgebraicGeometryIIGeometricStacksandApplications}, \cite[Proposition 2.4.2]{GaRoI}, and so there is an induced map \begin{equation}\label{Sheafification map}
    L(\mathcal{Q}) \to \LTd\sslash W
\end{equation} of fppf stacks.

\begin{Proposition}
    The map \labelcref{Sheafification map} is an isomorphism.
\end{Proposition}

\begin{proof} One can directly compute that the map $\Phi_0: \LTd \to \LTd\sslash W$ is a finite flat morphism using \cref{For Pseudo-Reflection Groups Acting Faithfully on Vector Space The GIT Quotient is Flat}. In particular, $\Phi_0$ is closed. Moreover, $\Phi_0$ is dominant by \cite[Lemma 29.8.8]{StacksProject} since the associated map of rings is injective. Therefore $\Phi_0$ is a flat surjection. It is also a morphism of finite type schemes by say \cref{Chevalley-Shephard-Todd Theorem}. Therefore $\Phi_0$ is a fppf surjection. Thus by \cite[Lemma 2.3.8]{GaRoI} (see also \cite[Section 2.3.2]{GaRoI}) the induced morphism obtained from the \v Cech nerve of $\Phi_0$ is an equivalence. Using \cref{Union of Graphs is Product of LieTs over GIT Quotient}, it is not difficult to show that the \v Cech nerve of $\Phi_0$ is equivalently given by the groupoid $\Gamma_W^{\bullet}$ constructed above. We deduce that the map $\mathcal{Q} \to \LTd\sslash W$ is a fppf equivalence. Now, since a map is a fppf equivalence if and only if the indued map on (fppf) sheafifications is an isomorphism (see \cite[Section 2.3.4]{GaRoI} and \cite[Section 2.3.2]{GaRoI}) we see that the induced map \labelcref{Sheafification map} is an equivalence, as desired. 
\end{proof}

One disadvantage of taking the fppf sheafification $L(\LTd\sslash\Wext)$ of $\LTd\sslash\Wext$ as the definition of the coarse quotient is that it is not immediately clear whether $L(\LTd\sslash\Wext)$ remains 0-coconnective or locally of finite type; in other words, it is not clear that the analogue of the second sentence in \cref{Affine Coarse Quotient is 0-Coconnective LFT} holds after sheafification.\footnote{On the other hand, one can use the fact that sheafification commutes with finite limits \cite[Lemma 2.3.6]{GaRoI} (see also \cite[Section 2.3.2]{GaRoI} to deduce all other results of \cref{Definition of Coarse Quotient Subsubsection} when $\LTd\sslash\Wext$ is replaced with its sheafification.} In particular, it is not immediate that $L(\LTd\sslash\Wext)$ is locally almost of finite type, and thus the category $\IndCoh(L(\LTd\sslash\Wext))$ need not be defined since the functor $\IndCoh(-)$ is only defined on prestacks which are locally almost of finite type. On the other hand, since $\IndCoh(-)$ satisfies fppf descent, if $L(\LTd\sslash\Wext))$ were locally of almost finite type, one can use Lurie's result that $\IndCoh(-)$ satisfies fppf descent \cite[Theorem 8.3.2]{GaiIndCoh} to prove that pullback by the sheafification map gives an equivalence \[\IndCoh(L(\LTd\sslash\Wext)) \xrightarrow{\sim} \IndCoh(\LTd\sslash\Wext).\] Since our main goal is to study the category $\IndCoh(\LTd\sslash\Wext)$, we content ourselves with taking the (non-sheafified) \cref{Definition of Coarse Quotient} as our definition of $\LTd\sslash\Wext$. 
\subsubsection{Base Change for Quotients}
We now proceed to study the category $\IndCoh(\LTd\sslash\Wext)$. By \cref{Affine Coarse Quotient is 0-Coconnective LFT}, we obtain that $\IndCoh(\LTd\sslash\Wext)$ is defined. Moreover,  \cref{WAffIndFiniteFlat} implies that $\quotientmapforcoarsequotient$ is in particular ind-schematic and ind-proper surjection between laft prestacks. We now record two immediate consequences of this fact.

\begin{Corollary}\label{Definition of Pushing Forward by Quotientmapforcoarsequotient}\cite[Chapter 3, Section 0.1.2]{GaRoII}
The functor $\quotientmapforcoarsequotient^!$ admits a left adjoint satisfying base change against $!$-pullbacks.
\end{Corollary}

We denote this left adjoint by $\quotientmapforcoarsequotient_*^{\IndCoh}$.

\begin{Corollary}\cite[Chapter 3, Section 0.4.3]{GaRoII}
If $\quotientmapforcoarsequotient: \LTd \to \LTd\sslash\Wext$ denotes the quotient map, the pullback functor $\quotientmapforcoarsequotient^!: \IndCoh(\LTd\sslash\Wext) \to \IndCoh(\LTd)$ induces an equivalence 
\raggedbottom
\[\quotientmapforcoarsequotient^!: \IndCoh(\LTd\sslash\Wext) \xrightarrow{\sim} \text{Tot}(\IndCoh(\mathfrak{t}^{\ast \bullet}))\]

\noindent where $\mathfrak{t}^{\ast \bullet}$ is the cosimplicial prestack given by the \v{C}ech nerve of $\quotientmapforcoarsequotient$. 
\end{Corollary}

\begin{Remark}
The definition of the coarse quotient $\LTd\sslash\Wext$ was inspired by the definition of the coarse quotient given when $\Wext$ is replaced by a Coxeter group in \cite[Section 2.7.3]{BZG}.
\end{Remark}

\subsubsection{t-Structure for Sheaves on Quotients}\label{t-strucure on Sheaves on Stack and Coarse Quotient Subsubsection}
\newcommand{\quotientFromModCharacterLatticeToSslashWext}{q}
\newcommand{\mapfromGammatoGammaModCharacters}{\phi}
\newcommand{\mapfromGammaModCharactersToLTd}{\mathring{t}}
We use the following proposition define a $t$-structure on the category $\IndCoh(\LTd\sslash\Wext)$ and determine some of its basic properties in \cref{t-Exactness of all Quotientmapforcoarsequotient functors}. 
\begin{Proposition}\label{!-Pullback to Union of Graphs from Lie T and pushforward are t-Exact}
The maps $s^!: \IndCoh(\LTd) \to \IndCoh(\Gamma_{\Waff})$ and $s_*^{\IndCoh}$ are $t$-exact.
\end{Proposition}

\begin{proof}
Because the $t$-structure on $\IndCoh(\Gamma_{\Waff})$ is by definition compatible with filtered colimits \cite[Chapter 3, Section 1.2.1]{GaRoII} it suffices to show this claim when $\Wext = \Waff$. Note that the map $s_*^{\IndCoh}$ is $t$-exact because $s$ is ind-affine \cite[Chapter 3, Lemma 1.4.9]{GaRoII}, and therefore since $s$ is ind-proper, we have that $s^!$ is the right adjoint to the $t$-exact functor $s_*^{\IndCoh}$ (see \cref{Definition of Pushing Forward by Quotientmapforcoarsequotient}) and therefore is left $t$-exact. We now show that $s^!$ is right $t$-exact. 

We first claim that, to show $s^!$ is right $t$-exact, it suffices to show that $s^!(\mathcal{O}_{\LTd}) \in \IndCoh(\Gamma_{\Waff})^{\leq 0}$. This follows since $\IndCoh(\LTd)^{\leq 0}$ is equivalently smallest $\infty$-category of $\IndCoh(\LTd)$ containing $\mathcal{O}_{\LTd}$ and closed under colimits (which can be seen, for example, by the $t$-exact equivalence $\Psi_{\LTd}: \IndCoh(\LTd)^{\leq 0} \xrightarrow{\sim} \QCoh(\LTd)^{\leq 0}$ given by the fact $\LTd$ is smooth and classical). Since $s^!$ commutes with colimits and the subcategory $\IndCoh(\Gamma_{\Waff})^{\leq 0}$ is closed under colimits, we see that it remains to show that $s^!(\mathcal{O}_{\LTd}) \in \IndCoh(\Gamma_{\Waff})^{\leq 0}$.
In turn, to show this, we first note that
\raggedbottom
\[s^!(\mathcal{O}_{\LTd}) \simeq s^!(\omega_{\LTd}[-d]) \simeq \omega_{\Gamma}[-d] \simeq \text{colim}_m i_{m, *}^{\IndCoh}(\omega_{\Gamma_m})[-d]\]

\noindent where the first equivalence follows from the fact that $\LTd$ is smooth, the second follows from by the definition of the dualizing complex and the functoriality of $!$-pullback, and the third follows since we have an equivalence $\IndCoh(\Gamma) \xleftarrow{\sim} \text{colim}_m\IndCoh(\Gamma_m)$. We claim that each dualizing complex $\omega_{\Gamma_m}$ is concentrated in a single cohomological degree, i.e. $\Gamma_m$ is Cohen-Macaulay. This follows from the fact that the map $\Gamma_m \to \LTd$ is a finite flat map (\cref{FiniteFlat}) to affine space, and therefore $\Gamma_m$ is Cohen-Macaulay. Thus each object of the above colimit is contained entirely in cohomological degree zero \cite[Chapter 4, Lemma 1.2.5]{GaRoII} and therefore so too is $s^!(\mathcal{O}_{\LTd})$ since the $t$-structure is compatible with filtered colimits.  
\end{proof}
\newcommand{\quotientmaptostackquotient}{q}
Recall the canonical quotient map $\quotientmapforcoarsequotient: \LTd \to \LTd\sslash\Wext$. Define a $t$-structure on $\IndCoh(\LTd\sslash\Wext)$ by declaring $\IndCoh(\LTd\sslash\Wext)^{\leq 0}$ to be the full ordinary $\infty$-subcategory closed under colimits and containing $\quotientmapforcoarsequotient_*^{\IndCoh}(\mathcal{O}_{\LTd})$. Similarly, we define a $t$-structure on $\IndCoh(\LTd/\Wext)$ (respectively, $\IndCoh(\LTd/\characterlatticeforT)$) by declaring $\IndCoh(\LTd/\Wext)^{\leq 0}$ to be the full ordinary $\infty$-subcategory closed under colimits and containing the respective IndCoh pushforward given by the quotient map of the structure sheaf $\mathcal{O}_{\LTd}$. Note that these do indeed define $t$-structures since the inclusion functor preserves colimits, and therefore admits a right adjoint. We now record further properties of these $t$-structures: 

\begin{Proposition}\label{t-Exactness of all Quotientmapforcoarsequotient functors}With the $t$-structure on $\IndCoh(\LTd\sslash\Wext)$ defined as above, we have the following: 
\begin{enumerate}
    \item The functor $\quotientmapforcoarsequotient^!\quotientmapforcoarsequotient_*^{\IndCoh}$ is $t$-exact. 
    \item The map $\quotientmapforcoarsequotient^!$ is $t$-exact and reflects the $t$-structure.
    \item The map $\quotientmapforcoarsequotient_*^{\IndCoh}: \IndCoh(\LTd) \to \IndCoh(\LTd\sslash\Wext)$ is $t$-exact. 
    \item The $t$-structure on $\IndCoh(\LTd\sslash\Wext)$ is compatible with filtered colimits. 
\end{enumerate}
\end{Proposition}

\begin{proof}

The first claim follows by base change (\cref{Definition of Pushing Forward by Quotientmapforcoarsequotient}) of the Cartesian diagram in \cref{SpecificGroupoidResults}, since we may identify this functor with the composite of $s^!$, $t$-exact by \cref{!-Pullback to Union of Graphs from Lie T and pushforward are t-Exact}, with the functor $s_*^{\IndCoh}$, which is $t$-exact since $s$ is ind-affine \cite[Chapter 3, Lemma 1.4.9]{GaRoII}.

Next, we show that $\quotientmapforcoarsequotient^!$ is right $t$-exact. If $\mathcal{G} \in \IndCoh(\LTd\sslash\Wext)^{\leq 0}$ we may write $\mathcal{G}$ as some colimit $\text{colim} (\quotientmapforcoarsequotient_*^{\IndCoh}(\mathcal{O}_{\LTd}))$. Since $\quotientmapforcoarsequotient^!$ is continuous, we see that by (1) that $\quotientmapforcoarsequotient^!(\mathcal{G})$ is a colimit of objects in the heart of the $t$-structure, and thus lies in $\IndCoh(\LTd)^{\leq 0}$. 

To see the left $t$-exactness of $\quotientmapforcoarsequotient^!$, let $\F \in \IndCoh(\LTd\sslash\Wext)^{> 0}$. We wish to show that $\quotientmapforcoarsequotient^!(\F) \in \IndCoh(\LTd)^{> 0}$, and to show this it suffices to show that $\Hom_{\IndCoh(\LTd)}(\mathcal{O}_{\LTd}, \quotientmapforcoarsequotient^!(\F))$ vanishes. However, since $\quotientmapforcoarsequotient$ is ind-proper, we see that by adjunction (\cref{Definition of Pushing Forward by Quotientmapforcoarsequotient}) it suffices to show $\Hom_{\IndCoh(\LTd\sslash\Wext)}(\quotientmapforcoarsequotient_*^{\IndCoh}(\mathcal{O}_{\LTd}), \F)$ vanishes, which follows by the definition of the $t$-structure. Thus the functor  $\quotientmapforcoarsequotient^!$ is $t$-exact, and this along with its conservativity gives (2). 

Now, to show (3), note that (2) gives that $\quotientmapforcoarsequotient^!$ reflects the $t$-structure, so it suffices to show that $\quotientmapforcoarsequotient^!\quotientmapforcoarsequotient_*^{\IndCoh}$ is $t$-exact, which is precisely (1).  Finally, (4) follows from the fact that $\quotientmapforcoarsequotient^!$ is continuous and reflects the $t$-structure, along with the fact that $t$-structure on $\IndCoh(\LTd)$ is compatible with filtered colimits. 
\end{proof}

We can use a similar argument to construct $t$-structures in the setting of a discrete group acting on some ind-scheme. Let $X$ denote any discrete group acting on some ind-scheme $\Gamma$ and let $\mapfromGammatoGammaModCharacters: \Gamma \to \Gamma/X$ denote the quotient map. Define a $t$-structure on $\IndCoh(\Gamma/X)$ via setting $\IndCoh(\Gamma/X)^{\leq 0}$ to be the full subcategory generated under colimits by objects of the form $\mapfromGammatoGammaModCharacters_{\ast}^{\IndCoh}(\F)$ for $\F \in \IndCoh(\Gamma)^{\heartsuit}$ (or, equivalently by the continuity of $\mapfromGammatoGammaModCharacters_{\ast}^{\IndCoh}$, for $\F \in \IndCoh(\Gamma)^{\leq 0}$).  

\begin{Lemma}\label{t-Exactness of all Quotientmaptostackquotient functors}
The functors $\mapfromGammatoGammaModCharacters^{\IndCoh}_{\ast}$ and $\mapfromGammatoGammaModCharacters^!$ are $t$-exact.
\end{Lemma}

\begin{proof}
We first show $\mapfromGammatoGammaModCharacters^!$ is $t$-exact. Consider the Cartesian diagram:

\begin{equation}\label{Ind-Scheme With Action of Discrete Set of Points Cartesian Diagram}
  \xymatrix@R+2em@C+2em{
   \Gamma \times X\ar[r]^{\text{act}} \ar[d]^{\text{proj}} & \Gamma \ar[d]^{\mapfromGammatoGammaModCharacters} \\
  \Gamma \ar[r]^{\mapfromGammatoGammaModCharacters} & \Gamma/X
  }
 \end{equation}

\noindent given by the quotient. Since act is ind-affine, it is $t$-exact. We also have by direct computation that $\text{proj}^!$ is $t$-exact. Therefore, by base change and continuity of $\mapfromGammatoGammaModCharacters^!$, we see that $\mapfromGammatoGammaModCharacters^!$ is right $t$-exact. For left $t$-exactness, assume $\F \in \IndCoh(\Gamma/X)^{> 0}$ and $\mathcal{G} \in \IndCoh(\Gamma)^{\leq 0}$. Then since $\mapfromGammatoGammaModCharacters$ is ind-proper, by adjunction we see
\raggedbottom
\[\text{Hom}_{\IndCoh(\Gamma)}(\mathcal{G}, \mapfromGammatoGammaModCharacters^!(\F)) \simeq \text{Hom}_{\IndCoh(\Gamma/X)}(\mapfromGammatoGammaModCharacters_{\ast}^{\IndCoh}(\mathcal{G}), \F) \simeq 0\]

\noindent by construction of our $t$-structure, and so $\mapfromGammatoGammaModCharacters^!(\F) \in \IndCoh(\Gamma)^{> 0}$. Finally, since $\mapfromGammatoGammaModCharacters^!$ is $t$-exact and conservative, $\mapfromGammatoGammaModCharacters^{\IndCoh}_{\ast}$ is $t$-exact if and only if $\mapfromGammatoGammaModCharacters^!\mapfromGammatoGammaModCharacters^{\IndCoh}_{\ast}$ is $t$-exact. However, this follows from base change along diagram \labelcref{Ind-Scheme With Action of Discrete Set of Points Cartesian Diagram}. 
\end{proof}

\begin{Corollary}\label{t-Structure on Sheaves on Stack and Coarse Quotient is Both Left and right-complete}
The $t$-structures on $\IndCoh(\LTd/\Wext)$ and $\IndCoh(\LTd\sslash\Wext)$ are both left-complete and right-complete. 
\end{Corollary}

\begin{proof}
By \cref{t-Exactness of all Quotientmaptostackquotient functors} and \cref{t-Exactness of all Quotientmapforcoarsequotient functors}, both categories admit conservative, $t$-exact functors to $\IndCoh(\LTd)$ which commute with limits (since they are right adjoints). Any category which admits a conservative, $t$-exact functor which commutes with limits to a left-complete category is left-complete, therefore the left-completeness holds in this case, where $\IndCoh(\LTd)$ admits a $t$-exact equivalence to $\QCoh(\LTd)$ (since $\LTd$ is a smooth classical scheme) and therefore is left-complete. Similarly, each functor to $\IndCoh(\LTd)$ is continuous and so the right-completeness follows from the fact that $\IndCoh(\LTd)$ is also right-complete. 
\end{proof}

\newcommand{\quotientmapforFINITEcoarsequotient}{\overline{s}_{\text{fin}}}
\newcommand{\mapfromstackquotientofWexttocoarsequotientofWext}{\tilde{\phi}}
\newcommand{\mapfromstackquotientofWexttocoarsequotientofWextATAPOINT}{\dot{\tilde{\phi}}}
\subsection{Descent to the Coarse Quotient for Affine Weyl Groups}\label{CoxeterDescentSubsection}
The quotient map $\quotientmapforcoarsequotient: \LTd \to \LTd\sslash\Wext$ induces a canonical map of prestacks $\mapfromstackquotientofWexttocoarsequotientofWext: \LTd/\Wext \to \LTd\sslash\Wext$. We now study the pullback functor $\mapfromstackquotientofWexttocoarsequotientofWext^!$ and show that this functor behaves similarly to the case where $\Wext$ is replaced with a finite Weyl group. For example, we show that the functor $\mapfromstackquotientofWexttocoarsequotientofWext^!$ is fully faithful in \cref{CoherentFullyFaithful}. We define those sheaves in $\IndCoh(\LTd/\Wext)$ in the essential image of $\mapfromstackquotientofWexttocoarsequotientofWext^!$ as those sheaves \textit{descending to the coarse quotient} for $\Wext$, and provide descriptions of those sheaves descending to the coarse quotient $\LTd\sslash \Wext$ in \cref{Subsubsection with All Equivalent Conditions of Satisfying Descent to Wext Coarse Quotient} which parallel the description for the finite Weyl group case in \cref{Equivalent Conditions to Be in Essential Image of Pullback of GeneralStackToGITQuotient Pullback}.

\subsubsection{Fully Faithfulness of Affine Pullback}\label{Fully Faithfulness of Affine Pullback Subsubsection}
\newcommand{\SpecofClambda}{\text{Spec}(C_{\lambda})}
\newcommand{\SpecofClambdaprime}{\text{Spec}(C_{\lambda'})}
\renewcommand{\G}{\mathcal{G}}   
\begin{Theorem}\label{CoherentFullyFaithful}
The functor $\mapfromstackquotientofWexttocoarsequotientofWext^!$ is fully faithful. 
\end{Theorem}
This subsection will be dedicated to the proof of \cref{CoherentFullyFaithful}. For a given $x \in \LTd(\fieldpossiblydifferentfromgroundfield)$, let $C_x$ denote the coinvariant algebra for the action of $\Waff_x$ on $\LTd$, which is in particular a $\fieldpossiblydifferentfromgroundfield$-algebra. The closed subscheme $\Spec(C_x) \xhookrightarrow{} \LTd$ induces a map $\Spec(C_x)/\Waff_x \to \LTd/\Waff$ which we denote by $q|$. Furthermore, let $[x]$ denote the image of $x$ under the quotient map $\LTd \to \LTd/\characterlatticeforT$, and let $\overline{x}$ denote the image of $x$ under the quotient map $\quotientmaptostackquotient: \LTd \to \LTd/\Wext$. Since the map $\mapfromstackquotientofWexttocoarsequotientofWext$ induces a bijection on $K$-points, so we abuse notation in also regarding $\overline{x}$ as a $\fieldpossiblydifferentfromgroundfield$-point of $\LTd\sslash\Wext$. 

\begin{Proposition}\label{Cartesian Diagrams Involving Union of Graphs Proposition}
Fix some field-valued point $x \in \LTd(\fieldpossiblydifferentfromgroundfield)$. There is a $\Wext$-equivariant isomorphism 
\raggedbottom
\[\Gamma_{\Wext} \times_{\LTd} \text{Spec}(K) \simeq \Wext \mathop{\times}\limits^{\Waff} \coprod_{x' \in \text{orbit}_{\Waff}(x)} \text{Spec}(C_{x'})\]

\noindent and, moreover, the rectangles of the following diagram are (derived) Cartesian: 
\raggedbottom
\begin{equation*}
  \xymatrix@R+2em@C+2em{
  \Wext \mathop{\times}\limits^{\Waff} \coprod_{x' \in \text{orbit}_{\Waff}(x)} \Spec(C_{x'}) \ar[r]^{\text{\ \ \ \ \ \ \  }s} \ar[d]^{t} & \Spec(C_x)/\Waff_x \ar[d]^{\quotientmaptostackquotient|} \ar[r]^{\terminalmapfromCmodassociatedstabilizer} & \Spec(K) \ar[d]^{\overline{x}}\\
  \LTd \ar[r]^{\quotientmaptostackquotient} & \LTd/\Wext \ar[r]^{\mapfromstackquotientofWexttocoarsequotientofWext} & \LTd\sslash\Wext
  }
\end{equation*}
\end{Proposition}

\begin{proof}
We first claim the outer rectangle is Cartesian. Applying \cref{Action Map is Projection Map}, we may prove this first claim when $\Wext = \Waff$. Write $\Gamma_{\Waff}$ as a union of $\Gamma_S$ where $S \subseteq \Waff$ varies over the finite subsets. Because this set is filtered, colimits over it commute with all finite limits (and, in particular, Cartesian products), and so we obtain
\raggedbottom
\[\Gamma_{\Waff} \times_{\LTd} \SpecofL \simeq \coprod_S(\Gamma_S \times_{\LTd} \SpecofL) \simeq \coprod_S \coprod_{x' \in \text{orbit}(x)} (\Gamma_{S \cap \text{stab}(x')} \times_{\LTd} \SpecofL)\]

\noindent where the second equivalence follows from \cref{Picking a Field-Valued Point Splits Off Different Waff Orbits}. Therefore we see that the outer rectangle in \cref{Cartesian Diagrams Involving Union of Graphs Proposition} is Cartesian. 

The fact that the left box is Cartesian follows from the fact that the stack quotient $\LTd/\Wext$ is defined as the colimit over a groupoid $U$ such that $U_1 \simeq \Wext \times \LTd$, and so in particular $\LTd \times_{\LTd/\Wext} \LTd \simeq \Wext \times \LTd$ by \cref{GroupoidResults}. Now, because the outer rectangle is Cartesian and all of the maps are $\Wext$-equiariant, we may take the quotient by $\Wext$. This is a sifted colimit because the opposite category of the simplex category is sifted, and in particular, taking the quotient by $\Wext$ preserves the Cartesian product and shows the rightmost rectangle is Cartesian. 
\end{proof}

\begin{Lemma}\label{Left Adjoint of MapfromStackQuotientofWexttoCoarsequotientofWext Exists}
In the setup and notation of \cref{Cartesian Diagrams Involving Union of Graphs Proposition}, the functor $\mapfromstackquotientofWexttocoarsequotientofWext$ admits a (continuous) left adjoint.
\end{Lemma}

\begin{proof}
To show that $\mapfromstackquotientofWexttocoarsequotientofWext^!$ admits a left adjoint, it suffices to show that $\mapfromstackquotientofWexttocoarsequotientofWext^!$ commutes with (small) limits, by the adjoint functor theorem (see \cite[Chapter 5]{LuHTT}). To see that $\mapfromstackquotientofWexttocoarsequotientofWext^!$ commutes with small limits, consider the following commutative diagram:

\begin{equation}
  \xymatrix@R+2em@C+2em{
  \LTd \ar[r]^{\text{id}} \ar[d]^{\quotientmaptostackquotient} & \LTd \ar[d]^{\quotientmapforcoarsequotient}\\
  \LTd/\Wext \ar[r]^{\mapfromstackquotientofWexttocoarsequotientofWext}  & \LTd\sslash\Wext
  }
 \end{equation}
 
\noindent Since $\quotientmaptostackquotient$ is ind-proper and surjective on geometric points, by ind-proper descent (say) we have that $\quotientmaptostackquotient^!$ is conservative. Therefore we may check that a map in $\IndCoh(\LTd/\Waff)$ is an isomorphism after applying $\quotientmaptostackquotient^!$. However, since $\quotientmaptostackquotient^!\mapfromstackquotientofWexttocoarsequotientofWext^! \simeq \quotientmapforcoarsequotient^!$, we see that $\quotientmaptostackquotient^!\mapfromstackquotientofWexttocoarsequotientofWext^!$ commutes with small limits (since $\quotientmapforcoarsequotient^!$ is also a right adjoint since $\quotientmapforcoarsequotient$ is ind-proper by \cref{WAffIndFiniteFlat}) and so $\mapfromstackquotientofWexttocoarsequotientofWext^!$ commutes with small limits as well, and thus admits a left adjoint by the adjoint functor theorem \cite{LuHTT}. 
\end{proof}

Denote the left adjoint to $\mapfromstackquotientofWexttocoarsequotientofWext^!$ by $\mapfromstackquotientofWexttocoarsequotientofWext_*^{\IndCoh}$. This light abuse of notation is justified by the following: 

\begin{Corollary}\label{FakeBaseChangeLemma}
The three Cartesian diagrams of \cref{Cartesian Diagrams Involving Union of Graphs Proposition} satisfy base change. In particular, the canonical map $\terminalmapfromCmodassociatedstabilizer_*^{\IndCoh} \quotientmaptostackquotient|^! \to \overline{x}^!\mapfromstackquotientofWexttocoarsequotientofWext_*^{\IndCoh}$ is an isomorphism. 
\end{Corollary}

\begin{proof}
The left and the \lq large\rq{} Cartesian diagrams satisfy base change since the maps $\quotientmaptostackquotient$ and $\quotientmapforcoarsequotient$ are ind-schematic and so satisfy base change by \cite[Chapter 3, Theorem 5.4.3]{GaRoII}, see also \cref{Definition of Pushing Forward by Quotientmapforcoarsequotient}. To show base change for the other Cartesian diagram, note that we may check that the map is an isomorphism on a compact generator of $\LTd/\Waff$. We choose the generator $\mathcal{G} := \quotientmaptostackquotient_*^{\IndCoh}(\omega_{\LTd})$. The uniqueness of left adjoints then gives that $\mapfromstackquotientofWexttocoarsequotientofWext_*^{\IndCoh}(\mathcal{G}) \simeq \quotientmapforcoarsequotient_*^{\IndCoh}(\omega_{\LTd})$. Base change by the outer Cartesian diagram of \cref{Cartesian Diagrams Involving Union of Graphs Proposition} then gives the desired claim. 
\end{proof}

\begin{proof}[Proof of \cref{CoherentFullyFaithful}]
    To show that $\mapfromstackquotientofWexttocoarsequotientofWext^!$ is fully faithful, it suffices to show that the counit map $\mapfromstackquotientofWexttocoarsequotientofWext_*^{\IndCoh}\mapfromstackquotientofWexttocoarsequotientofWext^! \to \text{id}$ is an equivalence. Because $\quotientmapforcoarsequotient$ admits a left adjoint (\cref{Definition of Pushing Forward by Quotientmapforcoarsequotient}), we have that $\mathcal{G} := \quotientmapforcoarsequotient_*^{\IndCoh}(\omega_{\LTd})$ is a compact generator of $\IndCoh(\LTd\sslash\Waff)$. Therefore to show that the counit is an equivalence, by continuity it suffices to show that the map $c(\mathcal{G}): \mapfromstackquotientofWexttocoarsequotientofWext_*^{\IndCoh}\mapfromstackquotientofWexttocoarsequotientofWext^!(\mathcal{G}) \to \mathcal{G}$ is an equivalence. Since $\quotientmapforcoarsequotient$ is an ind-proper cover (see \cref{WAffIndFiniteFlat}), $\quotientmapforcoarsequotient^!$ is conservative, and so it suffices to show $s^!(c(\mathcal{G})): s^!(\mapfromstackquotientofWexttocoarsequotientofWext_*^{\IndCoh}\mapfromstackquotientofWexttocoarsequotientofWext^!(\mathcal{G})) \to s^!(\mathcal{G})$ is an equivalence. 
    
    However, our map $s^!(c(\mathcal{G}))$ is a map in $\IndCoh(\LTd)$, which is generated by the skyscraper sheaves associated to all field-valued points, a direct consequence of \cref{For any qcoh sheaf on smooth classical Noetherian scheme there exists a field valued point where the fiber doesn't vanish} and the smoothness of $\LTd$, which gives that $\Psi_{\LTd}$ is an equivalence. Therefore, we may show this map is an isomorphism when restricted to each field-valued point $x \in \LTd$. By \cref{FakeBaseChangeLemma}, we see that $\overline{x}^!c(\G) \simeq c_{\terminalmapfromCmodassociatedstabilizer}(\overline{x}^!(\G))$, where $c_{\terminalmapfromCmodassociatedstabilizer}$ denotes the counit of the adjunction $(\terminalmapfromCmodassociatedstabilizer_*^{\IndCoh}, \terminalmapfromCmodassociatedstabilizer^!)$ of for the finite Coxeter group. However, we have that $\terminalmapfromCmodassociatedstabilizer^!$ is fully faithful by \cref{Pullback From GIT to Stack Quotient of Finite Group of Affine Scheme is Fully Faithful}. Therefore, since this holds for every field-valued point $x$, $\mapfromstackquotientofWexttocoarsequotientofWext^!$ is also fully faithful.
\end{proof}

\subsubsection{Equivalent Characterizations of Descent to the Coarse Quotient for The Affine Weyl Group}\label{Subsubsection with All Equivalent Conditions of Satisfying Descent to Wext Coarse Quotient}
We have seen in \cref{CoherentFullyFaithful} that the functor $\mapfromstackquotientofWexttocoarsequotientofWext^!$ is fully faithful. In analogy with the case where $\Wext$ is replaced with a finite group, we make the following definition: 

\begin{Definition}\label{Descends To Coarse Quotient For Wext Definition}
We say that a sheaf $\F \in \IndCoh(\LTd)^{\Wext}$ \textit{descends to the coarse quotient} $\LTd\sslash \Wext$ if it lies in the essential image of $\mapfromstackquotientofWexttocoarsequotientofWext^!$. When the $\Wext$-action is clear from context, we will simply say the given sheaf \textit{descends to the coarse quotient}. 
\end{Definition}

We now provide many alternative characterizations of a $\Wext$-equivariant sheaf descending to the coarse quotient, noting that many of the following conditions involve the usual affine Weyl group $\Waff$ as opposed to the extended affine Weyl group $\Wext$. 

\begin{Theorem}\label{Various Conditions for Wext Equivariant Sheaf to Satisfy Coxteter Descent}
A sheaf $\mathcal{F} \in \IndCoh(\LTd)^{\Wext}$ descends to the coarse quotient $\LTd\sslash \Wext$ if and only if it satisfies one of the following equivalent conditions:
\begin{enumerate}
    \item For every field-valued point $x \in \LTd(K)$, the canonical $\Waff_x$-representation on $\overline{x}^!(\text{oblv}^{\Wext}_{\Waff_x}(\F))$ is trivial. 
    
    \item For every finite parabolic subgroup $W'$ of $\Waff$, the object $\text{oblv}^{\Wext}_{W'}(\F) \in \IndCoh(\LTd/W')$ descends to the coarse quotient $\LTd\sslash W'$. 
    
    \item The object $\text{oblv}^{\Wext}_{\langle r \rangle}(\F) \in \IndCoh(\LTd/\langle r \rangle)$ descends to the coarse quotient $\LTd\sslash \langle r \rangle$ for every reflection $r \in \Waff$.

    \item For each $n$, each cohomology group $\tau^{\geq n}\tau^{\leq n}(\F)$ given by the $t$-structure in \cref{t-strucure on Sheaves on Stack and Coarse Quotient Subsubsection} descends to the coarse quotient $\LTd\sslash \Wext$. 
\end{enumerate}

If $G$ is simply connected, then these conditions are moreover equivalent to the following conditions: 

\begin{enumerate}
    \item[(5)] The object $\text{oblv}^{\Wext}_{\langle s \rangle}(\F) \in \IndCoh(\LTd/\langle s \rangle)$ descends to the coarse quotient $\LTd\sslash \langle s \rangle$ for every simple reflection $s \in W$.
    
    \item[(6)] The object $\text{oblv}^{\Wext}_W(\F) \in \IndCoh(\LTd/W)$ descends to the coarse quotient $\LTd\sslash W$.

    %
\end{enumerate}
\end{Theorem}

\begin{proof}
We first show that $\F$ descends to the coarse quotient $\LTd\sslash \Wext$ if and only if $\F$ satisfies (1). For a given $x \in \LTd(K)$, note that the fact that the right box in \cref{Cartesian Diagrams Involving Union of Graphs Proposition} commutes implies that any object in the essential image of $\mapfromstackquotientofWexttocoarsequotientofWext^!$ has the property that the pullback to $\IndCoh(\Spec(K))^{\Waff_x} \simeq \text{Rep}_{K}(\Waff_x)$ is trivial. Therefore, it remains to show that the left adjoint $\mapfromstackquotientofWexttocoarsequotientofWext_*^{\IndCoh}$ of \cref{Left Adjoint of MapfromStackQuotientofWexttoCoarsequotientofWext Exists} is conservative on this subcategory, since a functor with an adjoint is an equivalence if and only if it is fully faithful and its adjoint is conservative. 

Assume we are given some nonzero $\F \in \IndCoh(\LTd)^{\Wext}$ has the property that, for every field-valued point $x$ of $\LTd$, the pullback to $\IndCoh(\Spec(K))^{\Waff_x} \simeq \text{Rep}_{K}(\Waff_x)$ is trivial. Since $\F$ is nonzero, its pullback $\quotientmaptostackquotient^!(\F)$ is nonzero, and so in particular there exists a field-valued point for which $x^!\quotientmaptostackquotient^!(\F) \simeq \overline{x}^!(\F)$ is nonzero by \cref{For any qcoh sheaf on smooth classical Noetherian scheme there exists a field valued point where the fiber doesn't vanish}. Furthermore, since $\Upsilon_{\LTd}$ is an equivalence, we see that $\overline{x}^!(\text{oblv}^{\Wext}_{\Waff_x}(\F))$ lies in the full subcategory determined by the fully faithful (see \cref{Pointwise Essential Image of Pullback to Equivariant Sheaves on Coinvariant Algebra for IndCoh}) functor $\Xi_{\text{Spec}(C)}^H$ since $\Upsilon$ intertwines with pullback. Therefore, by \cref{Pointwise Essential Image of Pullback to Equivariant Sheaves on Coinvariant Algebra for IndCoh}, we see that the assumption that $\overline{x}^!(\text{oblv}^{\Wext}_{\Waff_x}(\F))$ is the trivial representation implies that, in the notation of \cref{Cartesian Diagrams Involving Union of Graphs Proposition}, $q|^!(\text{oblv}^{\Wext}_{\Waff_x}(\F))$ lies in the essential image of $\terminalmapfromCmodassociatedstabilizer^!$. Moreover, this sheaf is nonzero since $\overline{x}^!(\F)$ is nonzero and, since $i_x: \Spec(K) \to \Spec(C_x)$ is surjective on geometric points, $i_x^!$ is conservative \cite[Chapter 4, Proposition 6.2.2]{GaRoI}, and therefore the pullback functor $i_x^!$ is conservative. Thus we in particular see that $\terminalmapfromCmodassociatedstabilizer_*^{\IndCoh}(q|^!(\text{oblv}^{\Wext}_{\Waff_x}(\F)))$ is nonzero. Applying base change (\cref{FakeBaseChangeLemma}), we therefore see that $\overline{x}^!\mapfromstackquotientofWexttocoarsequotientofWext_*^{\IndCoh}(\F)$ does not vanish, and therefore neither does $\mapfromstackquotientofWexttocoarsequotientofWext_*^{\IndCoh}(\F)$, as required. 

Now, to show that $(1) \Rightarrow (2)$, let $\F \in \IndCoh(\LTd)^{\Wext}$ be some sheaf satisfying (1) and assume $W'$ is some parabolic subgroup of $\Waff$. We wish to show that $\mathcal{G} := \text{oblv}^{\Wext}_{\Waff_x}(\F)$ descends to the coarse quotient $\LTd\sslash \Waff_x$. By \cref{Equivalent Conditions to Be in Essential Image of Pullback of GeneralStackToGITQuotient Pullback}(3), it suffices to show that the canonical $W'_x$-representation on $x^!(\mathcal{G})$ is trivial. However, note that the following diagram commutes
\raggedbottom
 \begin{equation*}
 \xymatrix@R+2em@C+2em{
 \text{Rep}(\Waff_x) \ar[r]^{\text{oblv}^{\Waff_x}_{W'_x}} & \text{Rep}(W'_x) \\
  \IndCoh(\LTd)^{\Waff_x} \ar[r]^{\text{oblv}^{\Waff_x}_{W'_x}} \ar[u]^{x^!} & \IndCoh(\LTd)^{W'_x} \ar[u]^{x^!}
  }
 \end{equation*}

\noindent and so the fact that the associated $\Waff_x$-representation structure on $x^!(\mathcal{G})$ is trivial implies that the associated $W'_x$-representation is trivial, as desired. 

Conversely, if we are given some $\F$ satisfying (2) and some field-valued $x \in \LTd(K)$, it is standard (see, for example, \cite[Proposition 5.3]{Lo}) that the subgroup $\Waff_x$ is a finite parabolic subgroup. Therefore we see that, by assumption, $\text{oblv}^{\Wext}_{\Waff_x}(\F)$ descends to the coarse quotient for $\LTd\sslash \Waff_x$, and so that by \cref{Equivalent Conditions to Be in Essential Image of Pullback of GeneralStackToGITQuotient Pullback}(3), the canonical $(\Waff_x)_x = \Waff_x$-representation on $x^!(\text{oblv}^{\Wext}_{\Waff_x}(\F))$ is trivial, as required. 

The equivalence $(2) \Leftrightarrow (3)$ follows directly from the fact that one can check if a given $\mathcal{G} \in \IndCoh(\LTd)^{W'}$ descends to the coarse quotient for $\LTd\sslash W'$ if and only if $\text{oblv}^{W'}_{\langle r \rangle}(\mathcal{G}) \in \IndCoh(\LTd)^{\langle r \rangle}$ descends to the coarse quotient $\LTd\sslash \langle r \rangle$ for all reflections $r \in W'$, see \cref{Equivalent Conditions to Be in Essential Image of Pullback of GeneralStackToGITQuotient Pullback}. Similarly, the equivalence $(5) \Leftrightarrow (6)$ by varying $r$ over all \textit{simple} reflections of $W$, see \cref{Equivalent Conditions to Be in Essential Image of Pullback of GeneralStackToGITQuotient Pullback}(3). 

The proof of the equivalence $(3) \Leftrightarrow (5)$ follows nearly identically to the proof of the claim \lq $(2) \Leftrightarrow (3)$\rq{} of \cref{Equivalent Conditions to Be in Essential Image of Pullback of GeneralStackToGITQuotient Pullback}. The relevant addition is the standard fact (see, for example, \cite[Lemma 2.1.1]{BezrukavnikovMirkovicSomersRepresentationsofSemisimpleLieAlgebrasinPrimeCharacteristic}) that, if $G$ is simply connected, any reflection of $\Waff$ is conjugate in $\Wext$ to some simple reflection of $W$.

Because $\mapfromstackquotientofWexttocoarsequotientofWext^!$ is fully faithful (\cref{CoherentFullyFaithful}) and $t$-exact (\cref{t-Exactness of all Quotientmapforcoarsequotient functors}) we have that the essential image is closed under truncations, thus showing that if $\F$ descends to the coarse quotient, then so too does its cohomology groups. Since the $t$-structure on $\IndCoh(\LTd/\Wext)$ is left-complete and right-complete (\cref{t-Structure on Sheaves on Stack and Coarse Quotient is Both Left and right-complete}), a given $\F \in \IndCoh(\LTd/\Wext)$ has a canonical isomorphism
\raggedbottom
\[\F \simeq \text{lim}_m\text{colim}_n\tau^{\geq m}\tau^{\leq n}(\F)\]

\noindent where $-m, n \in \mathbb{Z}^{\geq 0}$. Therefore, since $\mapfromstackquotientofWexttocoarsequotientofWext^!$ is a continuous right adjoint, its essential image is closed under both limits and colimits. Thus since the essential image of $\phi^!$ is also closed under extensions (by fully faithfulness) we see that if all cohomology groups of $\F$ descend to the coarse quotient $\LTd\sslash \Wext$, so too does $\F$. 
\end{proof}

\begin{Remark}
    Using a variant of the argument in \cref{Cant Just Check Closed Points}, we claim it is possible to prove that condition (1) of \cref{Various Conditions for Wext Equivariant Sheaf to Satisfy Coxteter Descent} is strictly stronger than the condition \begin{enumerate}
        \item[(1')] For every closed point $x \in \LTd(k)$, the canonical $\Waff_x$-representation on $\overline{x}^!(\text{oblv}^{\Wext}_{\Waff_x}(\F))$ is trivial. 
    \end{enumerate}

    We do not include a proof here.
\end{Remark}

\begin{Remark} If $G$ is not simply connected, it is possible for conditions (5) and (6) of \cref{Various Conditions for Wext Equivariant Sheaf to Satisfy Coxteter Descent} to hold for some $\F \in \IndCoh(\LTd)^{\Wext}$ which does not descend to the coarse quotient. We now give an example.

If $G := \mathrm{PGL}_2$, we may identify $\Wext \cong 2\mathbb{Z} \rtimes Z/2\mathbb{Z}$ and $\LTd \cong \A^1$ in such a way that the element $(2, 0)$ acts by translation by 2 and $(0, 1)$ acts by reflection about 0. Let \[i: \mathcal{X} :=\coprod_{\{2m + 1 : m \in \mathbb{Z}\}} \Spec(k) \xhookrightarrow{} \A^1\] denote the inclusion of the ind-closed subscheme of odd integers. Since this ind-closed subscheme is invariant under the action of $\Waff$, there is a natural $\Waff$-equivariant structure on $\omega_{\mathcal{X}}$ and thus on $i_*^{\IndCoh}(\omega_{\mathcal{X}})$. Let $\F := i_*^{\IndCoh}(\omega_{\mathcal{X}}) \otimes_{k}k_{\mathrm{sign}}$ denote the $\Wext$-equivariant sheaf obtained by twisting the equivariant structure on $i_*^{\IndCoh}(\omega_{\mathcal{X}})$ by the sign character of $\Wext$. Then $\F$ does not descend to the coarse quotient by \cref{Various Conditions for Wext Equivariant Sheaf to Satisfy Coxteter Descent} since, for example, the induced $\mathbb{Z}/2\mathbb{Z}$-representation on the !-restriction along the closed point $\{1\} \to \A^1$ is the sign representation. On the other hand, since the !-restriction of $\F$ to 0 vanishes, $\text{oblv}_{\mathbb{Z}/2\mathbb{Z}}^{\Wext}(\F)$ descends to $\LTd\sslash W$ by \cref{Trivial Restriction Characterization of Essential Image of Pullback by GeneralStacktoGITMap for Order Two Group}. 

\end{Remark}

\appendix 

\section{Review of Borel Isomorphism Extension}\label{Generalization of the Borel Isomorphism Section}

A celebrated theorem of Borel identifies the cohomology of the flag variety $H^*(G^{\vee}/B^{\vee})$ with the coinvariant algebra $C := \Symt/\Symt^W_+$, where $G^{\vee}$ denotes the Langlands dual group to $G$ and $B^{\vee}$ denotes the corresponding Borel subgroup. In this section, we review a statement of an upgrade of this theorem, \cref{ClosedGraphsTheorem},  when the flag variety is replaced with the (closed) Schubert variety $X_v \xhookrightarrow{} G^{\vee}/B^{\vee}$ and give an alternate proof. In fact, we will generalize this theorem to the cohomology of unions of Schubert cells $H^*(X_S)$ determined by \textit{closed} subsets of the Weyl group: 

\begin{Definition}\label{ClosedSubsetDefinition}
If $\tilde{W}$ is a Coxeter group, we say a subset $S \subseteq \tilde{W}$ is \textit{closed} if $w \in S$ and $w' \leq w$ implies $w' \in S$.  
\end{Definition}

To state our desired extension to closed subsets of the Weyl group, we first obtain the following alternate description of $C$. Consider the scheme $\LTd \times \LTd$, and let $\text{graph}(w)$ denote the closed subscheme cut out by the ideal $I_{\text{graph}(w)}$, defined in turn to be the ideal generated by elements of the form $wp \otimes 1 - 1 \otimes p$ for $p \in \Symt$. Set $I_W := \cap_{w \in W}I_{\text{graph}(w)}$, and set $J_W$ to be the ideal generated by $I_W$ and $\Symtplus \otimes \Symt$. Similarly, assume we are given a closed subset $S \subseteq W$. Set $I_S := \cap_{w \in S}I_{\text{graph}(w)}$, and set $J_S$ to be the ideal generated by $I_S$ and $\Symtplus \otimes \Symt$. Note that we obtain a canonical map 
\raggedbottom
\[\Phi: \Symt \otimes_{\Symt^W} \Symt \to \Symtt/I_W\]

\noindent We now prove the following proposition, which says that the product $\LTd \times_{\LTd\sslash W} \LTd$ may be identified with the union of graphs of $W$:

\begin{Proposition}\label{Union of Graphs is Product of LieTs over GIT Quotient}
The map $\Phi$ is an isomorphism. 
\end{Proposition}

\begin{proof}
We note that $\Phi$ is in particular a map of $\Symt$-modules. We see that by base changing the quotient map $\LTd \to \LTd\sslash W$ by itself, by \cref{For Pseudo-Reflection Groups Acting Faithfully on Vector Space The GIT Quotient is Flat} the $\Symt$-module $\Symt \otimes_{\Symt^W} \Symt$ is finite flat. The fact that $\Symt \otimes_{\Symt^W} \Symt$ is finite implies that $\Symtt/I_W$ is also finite as a $\Symt$-module. We also see that $W$ acts freely on a dense open subset of $\LTd$, so that for a dense open subset of $\LTd$, the fiber of the map $\text{Spec}(\Symtt/I_W) \to \LTd$ has degree $W$. This in particular implies that all fibers of the $\Symt$-module $\Symtt/I_W$ have rank no less than $|W|$, since degree of a finite morphism is a upper semicontinuous function on the target. However, we also see that each fiber at some point $x$ admits a surjection from the fiber of $\Symt \otimes_{\Symt^W} \Symt$ at the point, which has rank precisely $|W|$ since $\Symt$ is a free rank $\Symt^W$-module of rank $W$ by by the Chevalley-Shephard-Todd theorem.
\end{proof}

\begin{Remark}
We temporarily assume that $G_{\mathbb{Z}}$ is an adjoint type Chevalley group scheme defined over the integers with maximal torus $T_{\mathbb{Z}}$. An identical result to \cref{Union of Graphs is Product of LieTs over GIT Quotient} with a similar proof (which will not be used below) also holds for the action of the Weyl group $W$ on the torus $T_{\mathbb{Z}}$. The analogue of \cref{For Pseudo-Reflection Groups Acting Faithfully on Vector Space The GIT Quotient is Flat} is the Pittie-Steinberg theorem \cite{SteinbergOnATheoremofPittie}, which says that for such $G_{\mathbb{Z}}$ that $\mathcal{O}(T_{\mathbb{Z}}^{\vee})$ is a free $\mathcal{O}(T_{\mathbb{Z}}^{\vee})^W$-module of rank $|W|$. 
\end{Remark}

\begin{Remark}\label{UnionNotProductExample}
Note that we work with the union of graphs given by the intersection of ideals, not the product. For example, if $\LG = \mathfrak{sl}_3$, we may pick a simple reflection $s \in W$ and choose coordinates on $\LTd$ so that $\LTd \simeq \text{Spec}(k[h,p])$ where $s(h) = -h$ and $s(p) = p$. The intersection $I_1 \cap I_s$ contains the degree 1 polynomial $p \otimes 1 - 1 \otimes p$, whereas the product $I_1I_s$ is generated by degree two polynomials. 
\end{Remark}

In particular, \cref{Union of Graphs is Product of LieTs over GIT Quotient} shows that one can identify $C \cong \Symtt/J_W$. For a closed subset $S \subseteq W$, let $X_S$ denote the closed subvariety of $G^{\vee}/B^{\vee}$ given by the union of the Schubert cells labelled by $w \in S$. We now give an alternate proof of the following result of \cite{CarrellSomeRemarksOnRegularWeylGroupOrbitsandtheCohomologyOfSchubertVarieties}:

\begin{Theorem}\label{ClosedGraphsTheorem}
Fix a closed subset $S \subseteq W$. There is an isomorphism $H^*(X_S) \simeq \Symtt/J_{S}$ such that the following diagram commutes:

 \begin{equation*}
  \xymatrix@R+2em@C+2em{
  \Symtt/J_W \ar[r]^{\sim} \ar[d]_{} & H^*(G^{\vee}/B^{\vee})  \ar[d] \\
  \Symtt/J_{S} \ar[r]_{\sim} & H^*(X_S)
  }
 \end{equation*}
 
 \noindent where the vertical maps are the canonical quotient maps and the top arrow is the Borel isomorphism. 
\end{Theorem}


\begin{Remark}
An alternate description of the rings $\Symtt/J_S$ in the case $G = \text{GL}_n$ and $S$ is the closure of some Weyl group element is computed in \cite{ALPCohomologyOfSchubertSubvarietiesofGLnModP}. We thank Victor Reiner for making us aware of this reference. 
\end{Remark}

\subsubsection{Results on Demazure Operators}
We first will recall some definitions and results of \cite{BGGAnnouncement} and \cite{BGG}. 

\begin{Definition}
For a simple reflection $s \in W$ associated to a coroot $\alpha$, define the vector space map $D_s: \Symt \to \Symt$ via $D_s(f) := \frac{f - sf}{\alpha}$. For a $w \in W$, choose a reduced expression $w = s_{1}...s_{r}$ and set $D_w := D_{s_1}...D_{s_r}$. These $D_w$ are known as \textit{Demazure operators}. 
\end{Definition}

Let $w_0$ denote the longest element of the Weyl group $W$ with respect to some ordering, and let $\ell := \ell(w_0)$. We now recall the following theorem:

\begin{Theorem}\cite{BGG}\label{BGGSummary} We have the following: 
\begin{enumerate}
    \item The Demazure operators are well defined and independent of reduced expression. 
    \item If $s_1, ..., s_p$ is not a simple expression, then $D_{s_1}...D_{s_p}$ vanishes. 
    \item The Poincar\'e dual class to the Schubert variety $[X_1] \in H_0(G/B)$ maps to $\rho^{\ell}/\ell!$ in the coinvariant algebra, and if $S \subseteq W$ is closed, the vector space $H^*(X_S)$ has a basis given by the $D_{u}(\rho^{\ell}/\ell!)$ for which $w_0u^{-1} \in S$. 
\end{enumerate}
\end{Theorem}

\begin{Remark}
To translate between point (3) of \cref{BGGSummary} and Theorem 3.15 in \cite{BGG}, we note that the vector space $H^*(X_S)$ also, in the notation of \cite{BGG}, has basis $P_w$ for which $w \in S$. The notation $P_w$ will not be used outside this remark.
\end{Remark}
    
\subsubsection{Proof of \cref{ClosedGraphsTheorem}}
In this subsection, we prove \cref{ClosedGraphsTheorem}. 
To prove \cref{ClosedGraphsTheorem}, we will first determine a specific element of $\Symtt$ which projects to a nonzero homogeneous element of degree $\ell := \ell(w_0)$ under the composite $\Symtt \twoheadrightarrow \Symtt/J_W \simeq H^*(G/B)$. 


\begin{Proposition}\label{PolynomialProp}
There exists a polynomial $F(x,y) \in \Symtt$ such that 

\begin{enumerate}
\item $F(x,vx) = 0$ if $v \neq w_0$,
\item $F(x,w_0x) = \prod \gamma(x)$, where $\gamma$ varies over the positive coroots, and
\item $F(x,y) \neq 0$ in the coinvariant algebra $\Symtt / J_W$.
\end{enumerate}
\end{Proposition}

To prove \cref{PolynomialProp}, we will set the following notation, closely following the notation and proof of \cite[Theorem 3.15]{BGG}. 

\begin{Proposition}\label{Existence Of Q}
There exists some polynomial $Q(x, y) \in \Symtt$ of $y$-degree $\ell(w_0)$ for which $Q(x, wx) = 0$ for $w \neq w_0$ and such that $Q(x, w_0x)$ is generically nonvanishing.  
\end{Proposition}

\begin{proof}
There exists some polynomial $Q' \in \Symtt$ such that $Q'(x, wx) = 0$ for $w \neq wx$ and $Q'(x, w_0x)$ generically does not vanish. Set $R_{1 \times w_0}(y) := (1 \otimes \rho)^{\ell}/\ell!$, and for $w \neq w_0$ set $R_{1 \times w} := D_{1 \times w_0w^{-1}}(R_{1 \times w_0})$. These give minimal degree lifts of the basis of the coinvariant algebra labeled by the Schubert cells by \cref{BGGSummary}. In particular, there exist polynomials $g_w(x,y) \in \Symtt^{1 \times W}$ such that $Q'(x, y) = \sum_{w \in W}g_w(x,y)R_{1 \times w}(y)$. Set $Q(x, y) := \sum_{w \in W}g_w(x,x)R_{1 \times w}(y)$. Then since 

$$Q(x, vx) := \sum_{w \in W}g_w(x,x)R_{1 \times w}(vx) = \sum_{w \in W}g_w(x,vx)R_{1 \times w}(vx) = Q'(x, vx),$$

\noindent we see that $Q(x, wx) = 0$ for $w \neq w_0$ and $Q(x, w_0x)$ is generically nonvanishing. 
\end{proof}

Choose such a $g_w(x, y)$, $R_{1 \times w}(y)$, and $Q$ as in the proof of \cref{Existence Of Q}. Unfortunately, such a $Q$ need not satisfy condition (3) of \cref{PolynomialProp}, even if $Q'$ does. Therefore, we will need to modify our choice of $Q$ further. To this end, choose any reduced expression $w_0 = s_{\alpha_1}...s_{\alpha_{r}}$ labelled by coroots $\alpha_i$. Given this decomposition, set $w_i := s_{\alpha_i}...s_{\alpha_1}$, $v_i := s_{\alpha_{i + 1}}...s_{\alpha_r}$, $Q_i := D_{1 \times v_i}Q$, $\gamma_1 := \alpha_1$, and, for $i > 1$, we set\footnote{Note that the notation of \cite[Theorem 3.15]{BGG} contains a typographic error: in their version, $\alpha_1$ should be replaced with $\alpha_i$. This is reflected in their  \cite[Lemma 2.2]{BGG}, which is appealed to in the proof of \cite[Theorem 3.15]{BGG}} $\gamma_i := w_{i-1}^{-1}(\alpha_i)$.

\begin{Lemma}\label{BGG3.15}
For any $Q$ satisfying the conditions of \cref{Existence Of Q} and any reduced expression for $w_0$, in the notation above, each polynomial $Q_i$ has $y$-degree $i$, $Q_i(x, w_ix)\prod_{r \geq j > i} \gamma_j(x) = (-1)^{(r - i)}Q(x, w_0x)$ and $Q_i(x, wx) = 0$ if $w \ngeq w_i$.
\end{Lemma}

\begin{proof}
This proof closely follows the proof of the Lemma below \cite[Theorem 3.15]{BGG}; we include the details for the reader's convenience. We proceed by backward induction on $i$. Note that when $i = \ell(w_0)$, we see that by assumption $Q_i(x, wx) = Q(x, wx)$ so the claim follows trivially from \cref{Existence Of Q}. 

Now assume that the lemma has been proved for $Q_i$ for some $i>0$. Then we obviously have the $y$-degree of $Q_{i-1}$ is $i-1$. Furthermore, we compute that for any $w \in W$, we have
\raggedbottom
\begin{equation}\label{Deqtn}Q_{i-1}(x, wx) = \frac{Q_i(x, wx) - Q_i(x, s_{\alpha_{i-1}}wx)}{\alpha_i(wx)}.\end{equation}

\noindent In particular, if $w = w_{i-1}$, we see that by our inductive hypothesis, $Q_i(x, wx) = 0$, and furthermore that $\alpha_i(wx) = (w_{i-1}^{-1}\alpha_i)(x) = \gamma_i(x).$ Therefore, we see that in this case, we have
\raggedbottom
\[Q_{i-1}(x, wx) = -\frac{Q_i(x, s_{\alpha_{i-1}}wx)}{\gamma_i(x)}\]

\noindent and so by induction we have 
\raggedbottom
\[Q_{i-1}(x, w_{i-1}x)\prod_{r \geq j > i-1} \gamma_j(x) = (-1)^{(r - i)}Q(x, w_0x).\]

\noindent Finally, if $w \ngeq w_{i-1}$ \cite[Corollary 2.6]{BGG} implies that $w \ngeq w_i$ and $s_{\alpha_i}w \ngeq w_i$. By induction we see that both terms in the numerator of \cref{Deqtn} vanish so our claim is proved. 
\end{proof}

We note the following corollary of \cref{BGG3.15}:

\begin{Corollary}\label{CanFactorOut}
We have $g_{w_0}(x, x) \prod_{\gamma} \gamma(x) = (-1)^{\ell}Q(x, w_0x)$, and furthermore for all $w \in W$, $Q(x, w_0x)$ divides $g_w(x,x)\prod_{\gamma}\gamma(x)$.
\end{Corollary}

\newcommand{\tw}{\tilde{w}}
\begin{proof}
The first statement is a direct application of the $i = 0$ claim of \cref{BGG3.15}. We also use it as the base case of for the second statement, which we prove by backwards induction on $\ell(w)$ for $w \in W$. Fix $w \in W$, and set $\tw := w_0ww_0$. Apply $D_{1 \times \tw}$ to the equality $Q(x, y) := \sum_{u \in W}g_u(x,x)R_{1 \times u}(y)$ to obtain
\raggedbottom
\begin{equation}\label{DemazureKillsStuffEquation}
D_{1 \times \tw}(Q(x,y)) = \sum_{\{u : \ell(\tw w_0u^{-1}) = \ell(\tw ) + \ell(w_0u^{-1})\}}g_u(x,x)D_{1 \times \tw w_0u^{-1}}(R_{1 \times w_0})(y)
\end{equation}

\noindent where the other terms vanish by (2) of \cref{BGGSummary}. In particular, the set $\{u : \ell(\tw w_0u^{-1}) = \ell(\tw ) + \ell(w_0u^{-1})\}$ contains a unique element of minimal length, namely $u = w_0\tw w_0 = w$, because the element $y \in W$ of largest length for which $\ell(\tw y) = \ell(\tw) + \ell(y)$ is $y = \tw^{-1}w_0$. Multiply both sides of \cref{DemazureKillsStuffEquation} by $\prod_{\gamma}\gamma(x)$. By \cref{BGG3.15}, we have that $Q(x, w_0x)$ divides the left hand side, and by induction, we have that $Q(x, w_0x)$ divides all terms in the right hand side except the term $g_w(x,x)R_{1 \times w}(y)$, so therefore it divides $g_w(x,x)$. 
\end{proof}

\begin{proof}[Proof of \cref{PolynomialProp}]
Set $F(x,y) = \frac{Q(x,y)\prod \gamma(x)}{Q(x, w_0x)}$, which is well defined by \cref{CanFactorOut}. Then by construction, $F$ satisfies (2), and $F$ satisfies (1) because $Q$ does. Furthermore, since the coefficient on $R_{1 \times w_0}$ on $F$ is $(-1)^{\ell}$ in the coinvariant algebra, (3) is satisfied, thus completing the proof of \cref{PolynomialProp}.
\end{proof}

\begin{Corollary}\label{ExplicitBasis}
If $Z \subseteq W$ is a closed subset, the underlying vector space of $\Symtt/J_{Z}$ has basis given by the images of $D_{1 \times u}(F)$ for which $w_0u^{-1} \in Z$. 
\end{Corollary}

\begin{proof}
Consider the map $\Symtt/J_W \twoheadrightarrow \Symtt/J_{Z}$. Since, generically in $\LT$, the union of $|Z|$ graphs will have $|Z|$ points lying above them, we have that $\text{dim}_k(\Symtt/J_{Z}) \geq |Z|$. 

On the other hand, assume that $w_0u^{-1} \notin Z$. We will show that $D_{1 \times u}(F)(tx', x') = 0$ for all $t \in Z$. In particular, by \cref{PolynomialProp}, these elements are all linearly independent, and so by showing this we will obtain the opposite inequality $\text{dim}_k(\Symtt/J_Z) \leq |Z|$. Choose a reduced word decomposition for $w_0u^{-1}$ and of $u$ to obtain a reduced word decomposition of $w_0 = (w_0u^{-1})u$. Apply \cref{BGG3.15} (where, in the notation of the lemma, $i$ denotes the length of $w_0u^{-1}, w_i = uw_0, v_i = u$, and $w = t$) to see that $D_{1 \times u}(F)(x, t^{-1}x) = 0$ if $t^{-1} \ngeq uw_0$, where we make the coordinate change $x := tx'$. Since inversion preserves ordering, we see that this is equivalent to the condition that $t \ngeq w_0u^{-1}$. However, by assumption, $t \in Z$ and $Z$ is closed, so $t \ngeq w_0u^{-1}$, and so our desired vanishing holds. 
\end{proof} 

Setting $Z = S$ in \cref{ExplicitBasis}, this precisely matches the description of \cref{BGGSummary}, which therefore proves \cref{ClosedGraphsTheorem}.

\bibliographystyle{amsplain}
\bibliography{references}

\end{document}